\documentclass[ejs, preprint]{imsart}

\RequirePackage[OT1]{fontenc}
\RequirePackage{amsthm,amsmath}
\RequirePackage[numbers]{natbib}
\RequirePackage[colorlinks,citecolor=blue,urlcolor=blue]{hyperref}
\usepackage{mathrsfs}
\usepackage{mathtools}
\usepackage{amsfonts}
\usepackage{array}
\usepackage{longtable}
\usepackage{indentfirst}
\usepackage[utf8]{inputenc}
\usepackage{times}
\usepackage[mathscr]{eucal}
\usepackage[math]{easyeqn}
\usepackage{graphicx}
\usepackage{natbib}
\usepackage{algorithm2e}

\pubyear{2006}
\volume{0}
\issue{0}
\firstpage{1}
\lastpage{8}

\startlocaldefs
\numberwithin{equation}{section}
\theoremstyle{plain}
\newtheorem{theorem}{Theorem}[section]
\newtheorem{statement}{Statement}
\newtheorem{lemma}{Lemma}
\newtheorem{remark}{Remark}
\newtheorem{proposition}{Proposition}
\newtheorem{corollary}{Corollary}

\newcommand{\sqrtr}{\mathbin{\text{\rotatebox[origin=c]{90}{$C/\sqrt{n} \,\,$}}}}
\newcommand{\srtr}{\mathbin{\text{\rotatebox[origin=c]{90}{\text{ L. \ref{thm:gCompAppl}}}}}}
\newcommand{\approxtest}{
	\mathbin{\text{\rotatebox[origin=c]{-90}{$\underset{\sqrtr}{\overset{\srtr}{\approx}}$}}}}

\renewcommand{\(}{$\,}
\renewcommand{\)}{\,$}

\def\eqdef{\stackrel{\operatorname{def}}{=}}

\def\Lt{\text{L}(t)}
\def\Rt{\text{R}(t)}

\newcommand{\bb}[1]{\boldsymbol{#1}}

\renewcommand{\hat}[1]{\widehat{#1}}
\renewcommand{\tilde}[1]{\widetilde{#1}}

\renewcommand{\Gamma}{\varGamma}
\renewcommand{\Pi}{\varPi}
\renewcommand{\Sigma}{\varSigma}
\renewcommand{\Delta}{\varDelta}
\renewcommand{\Lambda}{\varLambda}
\renewcommand{\Psi}{\varPsi}
\renewcommand{\Phi}{\varPhi}
\renewcommand{\Theta}{\varTheta}
\renewcommand{\Omega}{\varOmega}
\renewcommand{\Xi}{\varXi}
\renewcommand{\Upsilon}{\varUpsilon}

\def\Var{\operatorname{Var}}

\def\argmin{\operatornamewithlimits{argmin}}

\def\tr{\operatorname{tr}}
\def\I{I\!\!I}
\def\R{I\!\!R}
\def\E{I\!\!E}
\def\P{I\!\!P}

\def\kappa{\varkappa}

\def\T{\top}
\def\diag{\operatorname{diag}}

\def\Qv{\mathbb{Q}}

\def\CONST{\mathtt{C}}

\def\mub{\overline{\mu}}
\def\b{\flat}

\usepackage{color}

\definecolor{blue(pigment)}{rgb}{0.2, 0.2, 0.6}
\definecolor{ultramarine}{rgb}{0.07, 0.04, 0.56}
\definecolor{darkspringgreen}{rgb}{0.09, 0.45, 0.27}
\definecolor{hookersgreen}{rgb}{0.0, 0.44, 0.0}
\definecolor{plum(traditional)}{rgb}{0.56, 0.27, 0.52}
\definecolor{purple(html/css)}{rgb}{0.5, 0.0, 0.5}
\definecolor{magenta(dye)}{rgb}{0.79, 0.08, 0.48}

\def\CONST{\mathtt{C} \hspace{0.1em}}

\def\qq{q}

\def\sbt{\hspace{1pt} \flat}

\def\ups{\bb{\upsilon}}

\def\xx{\mathtt{x}}

\def\zz{\mathfrak{z}}

\endlocaldefs

\begin{document}

\begin{frontmatter}
\title{Construction of Non-asymptotic Confidence Sets in \(2\)-Wasserstein Space}
\runtitle{Construction of confidence sets in \(2\)-W space}

\begin{aug}
	
\author{\fnms{Johannes} \snm{Ebert}
	\ead[label=e3]{johannes.ebert@hotmail.com}}

\address{Department of Economics\\
	 Humboldt University of Berlin\\
	Unter den Linden 6, 10099 Berlin, Germany\\
	\printead{e3}}	
	
\author{\fnms{Vladimir} \snm{Spokoiny}\ead[label=e1]{spokoiny@wias-berlin.de}}
\and
\author{\fnms{Alexandra} \snm{Suvorikova}\ead[label=e2]{suvorikova1@wias-berlin.de}}

\address{Weierstrass Institute for Applied Analysis and Stochastics\\
Mohrenstr. 39 10117 Berlin Germany\\
\printead{e1,e2}}

\runauthor{J. Ebert et al.}

\affiliation{Department of Economics, Humboldt University of Berlin}

\end{aug}

\begin{abstract}
In this paper, we consider a probabilistic setting where the probability measures are considered to be random objects. 
We propose a procedure of construction 
non-asymptotic confidence sets for empirical barycenters in \(2\)-Wasserstein space
and develop the idea further to construction of a non-parametric two-sample test
that is then applied to the detection of structural breaks in data with complex geometry. 
Both procedures mainly rely on the idea of multiplier bootstrap (\citet{spokoiny_zhilova_2015}, \citet{chernozhukov_2014}).
The main focus lies on probability measures that have commuting covariance matrices 
and belong to the same scatter-location family: 
we proof the validity of a bootstrap procedure that allows to compute confidence sets and critical values 
for a Wasserstein-based two-sample test.
\end{abstract}


\tableofcontents
\end{frontmatter}

\section{Introduction}
\label{seq:intro}
Many applications in modern statistics go beyond the scope of classic setting and 
deal with data which lie on certain manifolds: for instance, statistics on shape space, 
computer vision, medical image analysis, bioinformatics and so on. These problems have
a common feature, namely they are closely related to the detection of \textit{patterns}. 
Pattern is a very general concept that describes  some (unknown and hidden) structure in the data, which has to be revealed.
For instance, the problem of classification of neuro-cognitive states of mind is associated with
detection of brain activity patterns in fMRI~\citet{norman2006beyond}. Another example comes from
bioinformatics, namely from computational epigenetic, that aims to detect common patterns
in gene expression regulation~\citet{jaenisch2003epigenetic, bock2008computational}.
The latter one is supposed to be one of the crucial aspects of morphogenesis. 
Pattern can also be interpreted in a more specific way as a "typical" geometric shape inherent to all observed items.
Following the work by~\citet{kendall1977diffusion} we define shape as
whatever remains after proper normalization of the object (i.d. rotations, dilations, and shifts are factored out).
For example, the work~\citet{liu2010protein} estimates "typical" spatial configuration of protein backbones.
Basically, this setting appears in problems where the data is subjected to deformations through 
a random warping procedure. Such problems are also common for image analysis~\citet{amit_structural_1991, trouve_metamorphoses_2005} 
and shape analysis \citet{huckemann_intrinsic_2010}.

In what follows we consider the following probabilistic setting. Let \((X, \rho) \) be some general metric space
and \(\P \) a Borel measure on it. The straightforward generalisation of least-square estimator
leads to the concept of the Fr\'{e}chet mean~\citet{frechet1948elements}, that is the (set of) global minima of the variance
\[
x^* \subseteq \argmin_{x \in X}\int_{X}\rho^{2}(x, y)\P(dy),
\]
where \(x^* \) is referred to as the population Fr\'{e}chet mean that is not necessarily unique.
However, under certain settings it can be considered as the pattern induced by \(\P \).
Further we assume, that \(\P \) and \(X \) are such that \(x^* \) is unique. 
The issue is discussed in more details in Section~\ref{sec:barycenters}.
Given an iid sample \((y_1,..., y_n) \) s.t. \(y_i \sim \P \), one can build its empirical estimator
\[
{x}_{n} \eqdef \argmin_{x \in X}\frac{1}{n}\sum^n_{i = 1}\rho^2(x, y_i).
\]
There are several works~\citet{bhattacharya2003large, bhattacharya2005large, bhattacharya2008nonparametric} 
that present a detailed study of the asymptotic properties of the empirical Fr\'{e}chet mean in case \(X \) 
is a finite-dimensional differentiable manifold. The monograph~\cite{bhattacharya2008nonparametric} also 
describes the procedure of asymptotic confidence set construction for \(x_n \). 
In this work we consider a particular case of \((X, \rho) \), namely \(X =\mathcal{P}_2(\R^d)\) is the space of all probability measures
with finite second moment defined on \(\R^d \) and \(\rho =W_2 \) is \(2 \)-Wasserstein distance. 
To gain deeper knowledge in the structure of Wasserstein space and
the optimal transport theory we recommend two excellent books \citet{villani_optimal_2009, ambrosio2013user}.
Following the seminal paper~\citet{agueh2011barycenters}, 
from now on we refer to Fr\'{e}chet mean in Wasserstein space as the Wasserstein barycenter.
Thus, the population barycenter \(\mu^* \) and its empirical estimator \(\mu_n \), that is built using an observed iid sample
\(\{ \nu_1,..., \nu_n \} \) are defined as
\[
\mu^* \eqdef \argmin_{\mu \in \mathcal{P}_2(\R^d)}\int_{\mathcal{P}(\R^d)} W^2_2(\mu, \nu)\P(d\nu), \quad 
\mu_n \eqdef \argmin_{\mu \in \mathcal{P}_2(\R^d)}\frac{1}{n}\sum_{i = 1}^n W^2_2(\mu, \nu_i).
\]
Since its introduction, Wasserstein barycenter has become a popular tool in a variety of domains, 
including image processing \citet{solomon_convolutional_2015, rabin2011wasserstein}, 
and mathematical economics \citet{carlier_matching_2008}. 
\citet{bigot2012} provide a characterization of the population barycenter 
for various parametric classes of random transformations for probability measures with compact support. 
Recently, \citet{gouic_existence_2015} established the convergence of the empirical barycenter of an 
iid sample of random measures on a locally geodesic metric space towards its population barycenter.

A procedure, that allows to make statistical inference in Wasserstein space was introduced by \citet{loubesStatistical2015}. 
They consider a statistical deformation model and obtain the asymptotic distribution and a bootstrap procedure for the Wasserstein barycenter. 
They use the results to construct a goodness-of-fit test for the deformation model. 
However, their study is limited to probability measures on the real line. The similar setting is
discussed in~\citet{munk2015limit}. Authors study the subspace of Gaussian measures on \(\R^d \) and 
estimate Wasserstein distance \(W_{2}(\nu_1, \nu_2) \) between two Gaussians \(\nu_1 \) and \(\nu_2 \),
knowing its empirical counterparts \(\hat{\nu}_1\), \(\hat{\nu}_2 \). Empirical measures are estimated using iid samples 
\(X_1,..., X_n \sim \nu_1 \) and \(Y_1,..., Y_m \sim \nu_2 \), all \(X_i, Y_j \in \R^d \).

The present work sets out to generalize the results in \citet{loubesStatistical2015, munk2015limit} to the case
where random observed objects are measures on the space $\R^d$. 
Namely, we consider an iid sample \(\{\nu_1,..., \nu_n  \}\), \(\nu_i \sim \P \) and the following non-parametric test 
\[
T_n = \sqrt{n}W_2(\mu_n, \mu^*),
\]
and propose the procedure of construction of non-asymptotic data-driven confidence sets
\(\mathcal{C}\bigl(\zz^{\sbt}(\alpha) \bigr)\) around \( \mu_n\), s.t.
\[
\P \Bigl( \mu^* \not\in \mathcal{C}\bigl(\zz^{\sbt}(\alpha) \bigr) \Bigr) \leq \alpha.
\]
The procedure is based on the multiplier bootstrapping technique~\citet{spokoiny_zhilova_2015}, \citet{chernozhukov_2014}.

We use the same approach to construct a non-parametric two-sample test in \(2\)-Wasserstein space,
that is further applied to the problem of change point detection. 
The general statement is as follows: let \(\nu_t \) be an observed process in discrete time. 
The time moment \(t^* \) is supposed to be a change point if
the data stream in hand undergoes some abrupt structural break:
\[
\begin{cases}
\nu_t \sim \P_1,  & t < t^*,\\
\nu_t  \sim \P_2, & t \geq t^*
\end{cases}
\]
The goal is to detect the regime switch as soon as possible under given false-alarm rate.
We use a detection procedures that is based on a test in running window. 
 Let \( t \) be a candidate for a change point 
and let \( ({\nu}_{t - h},..., {\nu}_{t + h - 1}) \) be observed data in the rolling window of size \( 2h \).
Let
\[
\begin{cases}
\nu_i \sim \P_1, & i \in\{t - h,...,t - 1\}, \\ 
\nu_i \sim \P_2, & i \in\{t,...,t + h - 1\}. 
\end{cases} 
\]
Then the hypothesis of homogeneity \(H_0 \) and its alternative \(H_1 \) are written as
\[
H_0: \P_1 = \P_2, \qquad H_1: \P_1 \neq \P_2.
\]
As a test statistic we use
\[
T_h(t) \eqdef \sqrt{h}W_2 \bigl(  \mu_{l}(t), \mu_{r}(t)  \bigr),
\]
where \( \mu_l(t) \) and \( \mu_{r}(t) \) are empirical barycenters in the
left and right halves of a scrolling window:
\[
\mu_l(t) \eqdef \argmin_{\mu \in \mathcal{P}_2(\R^d)}\frac{1}{h}\sum_{i = t - h}^{h - 1} W^2_2(\mu, \nu_i), \quad
\mu_r(t) \eqdef \argmin_{\mu \in \mathcal{P}_2(\R^d)}\frac{1}{h}\sum_{i = t }^{t + h - 1} W^2_2(\mu, \nu_i).
\]
A change point is supposed to be detect at time moment \(t\) if \(T_h(t) \) exceeds some critical level:
\[
T_h(t) \geq \zz_h(\alpha, t).
\]
The crucial step of the method is the fully data-driven calibration of critical values \(\zz_h(\alpha, t) \),
that is also based on the idea of multiplier bootstrap.

As a starting point, we restrict the discussion to the case of scatter-location family of measures, 
for which an explicit representation of the Wasserstein distance exists~\citet{alvarez2011wide}. 
Section~\ref{section:confidence_sets} presents the procedure of construction of non-asymptotic confidence sets for
empirical Wasserstein barycenters. Section~\ref{section:change_point_wasserstein} studies its application to 
detection of structural breaks in data with complex geometry.
Theoretical justification is obtained for measures with commuting covariance matrices.
Both algorithms are tested under the most general setting on artificial and real data (MNIST database of hand written digits).
The results are presented in Section~\ref{section:experiments}. All proofs and conditions are collected in Appendix~\ref{chapter:proofs_barycenters}.

\section{Monge-Kantorovich distance for location-scatter family}
\label{sec:barycenters}
In this section we recall the basics on optimal transportation theory. Consider Euclidean space {\( (\R^d, \| \cdot \|_2) \)}
and denote as \(\mathcal{P}_2(\R^d) \) the set of all probability measures on \(\R^d \) with finite second moment. 
The space is endowed with \(2\)-Wasserstein distance, that is the solution of Monge-Kantorovich 
problem in a particular case of cost function \(c(x, y) = \| x - y\|^2 \). 
Namely, for any \( \mu\) and \(\nu \) in \(\mathcal{P}(\R^d) \) we define the distance as
\begin{equation}
\label{def:W2_distance}
W_2^2(\mu, \nu) \eqdef \inf_{\pi \in \Pi(\mu, \nu)} \int_{\R^d} \| x- y\|^2 d\pi(x,y),
\end{equation}
where \( \Pi(\mu, \nu) \) is the set of all joint probability measures \(\pi \in \mathcal{P}(\R^d \times \R^d)\) with marginals \( \mu \) and \(\nu \)
\[
\Pi(\mu, \nu) \eqdef \bigl\{ \pi \in \mathcal{P}(\R^d \times \R^d): 
~\text{for all}~ \mathcal{A} \in {\mathcal{B}(\R^d)} : \]
\[
\pi(\mathcal{A}\times\R^d) = \mu(\mathcal{A}), \pi(\R^d \times \mathcal{A}) = \nu(\mathcal{A}) \bigr\},
\]
where \(\mathcal{B}(\R^d) \) is Borel \(\sigma\)-algebra on \(\R^d \).
Given a sample \(\{\nu_1,..., \nu_n\} \) from \(\mathcal{P}_2(\R^2) \),
we define its empirical barycenter (see~\citet{agueh2011barycenters}) as 
\begin{equation}
\label{def:empirical_barycenter}
\mu_n \eqdef \argmin_{\mu \in \mathcal{P}_2(\R^d)} \frac{1}{n}\sum_{i = 1}^n W^2_2(\mu, \nu_i).
\end{equation}
An example is presented at Fig.~\ref{fig:barycenter}. 
The upper panel depicts an observed sample of \(n=4 \) normalized images of size \(100\times 100 \) pixels.
The left-hand lower box stands for Euclidean averaging of images, whereas the right-hand one
for \(2\)-Wasserstein barycenter, computed with Bregman projection
algorithm proposed in~\citet{benamou2015iterative}.

Application of the concept of weighted mean generalizes Wasserstein barycenter as follows: 
let the vector of weights \((w_1,..., w_n) \) be an element of a unit \(n\)-dimensional simplex, 
i.d. \(\sum_{i}w_i = 1 \) and \(w_i \geq 0 \) for all \(i = 1,..., n \). Then the weighted barycenter is 
\[
\mu_n \eqdef \argmin_{\mu \in \mathcal{P}_2(\R^2)} \sum_{i = 1}^n w_iW^2_2(\mu, \nu_i).
\]
Its existence, uniqueness and regularity are investigated in \citet{agueh2011barycenters}.
\begin{figure}[ht!]
	\centering
	\includegraphics[width=.75\textwidth]{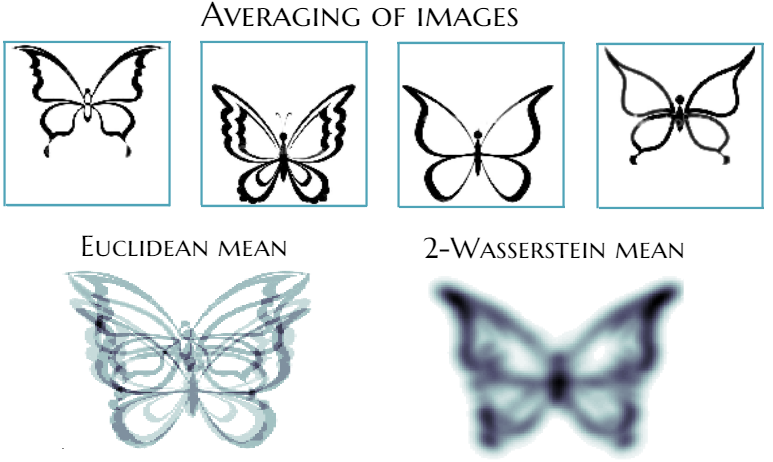}
	\caption{\(2\)-Wasserstein barycenter of 4 images} 
	\label{fig:barycenter}
\end{figure}

This work lays its focus on the case of a random sample \(\{\nu_1,..., \nu_n\} \), 
where all observations are independent and follow some unknown distribution \(\P \), 
i.d. \(\nu_i \sim \P \). A measure \(\P_n \) induced by this sample
\[
\P_n \eqdef \frac{1}{n}\sum^{n}_{i = 1}\delta_{\nu_i}
\]
can be considered as an empirical counterpart of \(\P \).

Furthermore, along with the empirical barycenter, one can define the \textit{population} one. Namely, this is 
a set of \(\{\mu^* \}\), that is
\begin{equation}
\label{def:population_barycenter_general}
\mu^* \subseteq \argmin_{ \mu \in \mathcal{P}_2(\R^d)} \int_{\mathcal{P}_2(\R^d)}W^2_2(\mu, \nu)d\P(\nu).
\end{equation}
Further we assume, that \(\P \) admits a unique barycenter. The next proposition ensures the fact.
\begin{proposition}[\citet{gouic_existence_2015}, Proposition 6]
Let \(\P \) be such that there exists a set \(A \subset \mathcal{P}_2(\R^d) \) of measures such that
for all \(\mu \in A\),
\[
B \in \mathcal{B}(\R^d), \quad \text{dim}(B) \leq d - 1 \longrightarrow \mu(B) = 0,
\]
and \( \P(A) > 0  \), then \(\P \) admits a unique barycenter.
\end{proposition}
The paper~\citet{gouic_existence_2015} shows, 
that {under accepted setting} \(\mu_n \) is a consistent estimator of \(\mu^* \).

\begin{proposition}[\citet{gouic_existence_2015}, Corollary 5]
	\label{proposition:gouic}
	Suppose that \(\P \in \mathcal{P} \bigl(  \mathcal{P}_2(\R^d)  \bigr) \) has a unique barycenter.
	Then for any sequence \(\bigl( \P_n \bigr)_{n \geq 1} \subset \mathcal{P} \bigl(  \mathcal{P}_2(\R^d)  \bigr) \)
	converging to \(\P \): \(W_2(\P, \P_n) \rightarrow 0\) as \(n \rightarrow \infty \),
	any sequence \(\mu_n \) of their barycenters converges to the barycenter of \(\P \):
	\[
	W_2(\mu^*, \mu_n) \rightarrow 0, ~~ n \rightarrow \infty.
	\]
\end{proposition}

Now we present the concept of location-scatter family. 
Let \( \mathcal{P}^{ac}_2(\R^d)\) be the set of all absolutely continuous probability measures on \(\R^d\)
and fix some template measure 
\( \mu_0 \in \mathcal{P}^{ac}_2(\R^d)\).
Consider a random variable \(X \) induced by the law \(\mu_0 \): \( X \sim \mu_0 \). 
A location-scatter family generated from \(\mu_0 \)
is a set of distributions that are generated by all possible positive definite affine transformations of \(X\).
\begin{description}
\item[ \( (\mathcal{F}) \)] 
\emph{	Let \(\mu_0 \in \mathcal{P}^{ac}_{2}(\R^d) \) be a template object, s.t. \(X \sim \mu_0 \), \
	\(\E X = m_0 \) and \(\Var(X) = \Qv_0 \). Let the family of transformations be}
	\[
	\mathcal{F}(\mu_0) = \bigl\{AX + a, \quad A \in  \mathbb{S}_{+}(d, \R),\quad a\in \R^d \bigr\},
	\]
\emph{ where \(\mathbb{S}_{+}(d, \R) \) is the set of all positive definite symmetric matrices of size \(d\times d \)
	with real entries.}
\end{description}
 For example, the class of all \(d\)-dimensional Gaussians can be considered as
 a class of all affine transformations of standard normal distribution:
\begin{equation}
\mathcal{F}\bigl(\mathcal{N}(0, I_{d})\bigr) = \bigl\{ \mathcal{N}(m, S): \quad m \in \R^d, \quad S = A^2  \bigr\}.
\end{equation}
Furthermore, a nice fact about \(2\)-Wasserstein distance between 
any two measures from the class \(\mathcal{F}\bigl(\mathcal{N}(0, I_{d})\bigr)\)
is that, it is completely defined by their first and second moments (see e.g. \citet{gelbrich_formula_1990, olkin_pukelsheim_1982}). 
Let \(\mu = \mathcal{N}(m_1, S_1) \) and \(\nu = \mathcal{N}(m_2, S_2) \).
Then the minimum in~\eqref{def:W2_distance} turns into
\begin{equation}
\label{def:wdist_gaussian}
W_2^2(\mu, \nu) = \|m_1 - m_2\|^2 + \tr\Bigl( S_1 + S_2 -2\bigl(S_1^{1/2}S_2S_1^{1/2} \bigr)^{1/2}\Bigr).
\end{equation}
Note, that simple calculations presented
in Statement~\ref{statement:W2_distance_gauss} show, that it can be rewritten as
\begin{equation*}
W_2^2(\mu, \nu) = \|m_1 -m_2 \|^2 +\bigl \| \bigl(A(S_1, S_2) - I \bigr)S^{1/2}_1 \bigr\|^2_{\text{F}}, 
\end{equation*}
where \( \|\cdot \|_{F} \) stands for Frobenius norm and 
\(A(S_1, S_2) \) is a symmetric positive-definite matrix, that is associated with the 
optimal linear map from \(\mathcal{N}(0, S_1) \) to \(\mathcal{N}(0, S_2) \)
\[
A(S_1, S_2) \eqdef S^{-1/2}_1\bigl(S^{1/2}_1S_2S^{1/2}_1 \bigr)^{1/2}S^{-1/2}_1.
\]
The paper \citet{alvarez2011wide} expands the result by~\citet{gelbrich_formula_1990} and shows 
that \(2\)-Wasserstein distance between any two measures from the same scatter-location family 
has the same form as~\eqref{def:wdist_gaussian}. 
Furthermore, it also generalizes the result by~\citet{agueh2011barycenters}, 
which claims that an empirical barycenter of any set of Gaussians is Gaussian as well. 
In particular, let \(\nu_1 = \mathcal{N}(m_1, S_1),..., \nu_n = \mathcal{N}(m_n, S_n) \) 
then the barycenter \(\mu_n \) is Gaussian with parameters \(\mathcal{N}(r_n, \Qv_n) \), where
\[
r_n = \frac{1}{n}\sum_{i = 1}^n m_i, \quad \Qv_n = \frac{1}{n}\sum_{i = 1}^n \bigl(\Qv^{1/2}_n S_i \Qv^{1/2}_n  \bigr)^{1/2}.
\]
The next proposition shows, that any scatter-location family is closed for barycenters as well.
\begin{proposition}[\citet{alvarez2011wide}, Theorem 3.11]
	\label{proposition:alvarez}
Let \( \mu_0 \in \mathcal{P}^{ac}_2(\R^d)  \) and \(\P \in \mathcal{P} \bigl(  \mathcal{P}_2(\R^d)  \bigr) \). 
Furthermore, we assume that \(\text{supp}(\P) = \mathcal{F}(\mu_0) \). 
In other words each observation \(\nu_{\omega} \sim \P \) and \(\nu_{\omega} \in \mathcal{F}(\mu_0) \), 
with first and second moments \( (m_{\omega}, S_{\omega}) \)  respectively. 
Then \(\mu^*  \) defined in~\eqref{def:population_barycenter_general} is 
the unique barycenter of \(\P \)
characterized by first and second moments that are defined as
\[
r^* = \int_{\R^d}m_{\omega}d\P(\omega),
\]
\[
\mathbb{Q}^* = \int_{S_{+}(d, \R)}\Bigl( \mathbb{Q}^{*1/2} S_{\omega} \mathbb{Q}^{*1/2}  \Bigr)^{1/2}d\P(\omega).
\]
Moreover, \(\mu^* \in \mathcal{F}(\mu_0) \).
\end{proposition}
Further we consider the following data-generating scheme. 
Let \( \mu_0 \in \mathcal{P}^{ac}_2(\R^d) \) be a fixed template measure 
and \(X \sim \mu_0 \):
\begin{equation}
\label{def:admiss_transf}
\mathcal{F}(\mu_0) = \bigl\{\nu_{\omega},~\text{s.t}~ Y_{\omega}\sim \nu_{\omega} ~\text{and}~ Y_{\omega} = A_{\omega}X + a_{\omega}, (A_{\omega}, a_{\omega}) \sim \P \bigr\}
\end{equation}
where
\[
\E_{\omega}Y_{\omega} = X,~\text{i.d.}~\quad \E_{\omega}A_{\omega} = I_{d}, \quad \E_{\omega}a_{\omega} = 0.
\]
One can easily derive that \(\mu_0 \) coincides with \( \mu^* \) ({see Statement~\ref{statement:population_barycenter_equiv}}).
Taking into account all aforementioned, one can consider an empirical barycenter \(\mu_n\)~\eqref{def:empirical_barycenter}  
as a good candidate for the template estimator. 
\begin{remark}
	Unless otherwise noted, from now on we refer to \(\mu_0 \) as \(\mu^* \).
	It is implicitly assumed, that we always talk about a population barycenter \(\mu^* \),
	that coincides with the template object \(\mu_0 \) in case of location-scatter family.
\end{remark}

The next section provides the procedure of construction
confidence sets around \(\mu_n \). A possible application to change point detection
is presented in Section~\ref{section:change_point_wasserstein}.

\section{Bootstrap procedure for confidence sets}
\label{section:confidence_sets}
First we present a general description of the procedure. And then specify the results for a particular case 
of the location-scatter families.
Let \(\{\nu_1,..., \nu_n \} \) be observed iid random sample, that comes from
distribution \(\P \). Let \(\mu_n \) and \(\mu^* \) be empirical and population barycenters respectively.
We define the following statistic based on the \(2\)-Wasserstein distance between them
\[
\label{def:W_stat}
T_n \eqdef \sqrt{n} W_{2}(\mu_n, \mu^*).
\]

A confidence set for the population barycenter is defined as
\begin{equation}
\label{def:conf_sets_barycenters}
\mathcal{C}(\zz)\eqdef \bigl\{\mu \in {\mathcal{P}_{2}(\R^d)} ~|~  \sqrt{n} W_{2}(\mu_n, \mu) \leq \zz \bigr\}.
\end{equation}
For \(\alpha \in (0,1)\) define the quantile \(\zz_{n}(\alpha)\) as the minimum value that ensures\\ \(\P\bigl(\mu^* \not\in \mathcal{C}(\mathtt{\zz}) \bigr)\leq  \alpha\) :
\[
\zz_{n}(\alpha) \eqdef \inf \bigl \{\zz \geq 0 \,| \, \P \bigl(T_n \geq \zz \bigr) \leq \alpha \bigr \}.
\]
The quantile \( \zz_{n}(\alpha) \) depends on the underlying distribution \(\P\), which is generally unknown. 
We therefore propose a weighted bootstrap procedure for the estimation of the quantiles of the statistic \(T_n\).
The idea is to mimic the  distribution of \(T_n\) by considering a weighted version of the barycenter problem, 
reweighing its summands with random multipliers. 
This leads to the following bootstrap version of the empirical barycenter
\begin{equation}
\label{def:barycenter_boot}
\mu^{\b}_n \eqdef \argmin_{\mu \in \mathcal{P}(\R^d)} \sum_{i=1}^{n} W_2^2(\mu, \nu_i) w_i,
\end{equation}
where the $w_i$ are iid random variables fulfilled condition \( {(E\!D^{W}_{\b})} \) from Section~\ref{section:conditions_barycenters}. 
The distribution of \( \mu^{\b}_n \) is conditional on the data \( \{ \nu_1,..., \nu_n\}\). 
In what follows, \(\P^{\b}\) denotes the distribution of the weights $w_i$ 
given the sample \( \{ \nu_1,..., \nu_n\}\). The counterpart of $T_n$ in the bootstrap world is defined as 
\begin{equation}
\label{def:test_statistic_boot}
T_n^{\b} \eqdef \sqrt{n}W_2(\mu^{ \b}_n, \mu_n).
\end{equation}
\begin{remark}[Choice of bootstrap weights]
	\label{remark:weights}
	Note, that if the weights are non-negative, e.g. \(w_i \sim \text{Po}(1) \) or \(w_i \sim \text{Exp}(1) \), 
	the bootstrapped barycenter \(\mu^{ \b}_n \) is unique and belongs to the scatter-location family \(\mathcal{F}(\mu_0) \). 
	Otherwise, e.g. if weights are normal \(w_i \sim \mathcal{N}(1,1) \), {the existence of the solution of} \eqref{def:barycenter_boot} 
	should be proven. In what follows we show, that at least in some cases the framework admissible
	for non-negative weights can be applied to negative ones as well.
\end{remark}
The bootstrap counterpart of the quantile \(\zz_{n}(\alpha)\) is defined as
\begin{equation}
\label{def:boot_quantile_barycenters}
\zz^{\sbt}_{n}(\alpha) \eqdef \inf\bigl\{\zz \geq 0 \,| \, \P^{\sbt}(T^{\sbt}_n \geq \zz) \leq \alpha \bigr \}.
\end{equation}
Note that this quantity depends on the sample and is therefore random. The procedure is presented in Algorithm~\ref{alg:computation_z}

\subsection{The case of commuting matrices}
In this section we consider the case, when all observations
\( \{ \nu_1,..., \nu_n \} \) come from the same scatter-location family \(\mathcal{F}^{C}(\mu_0) \)
of measures, that have commuting covariance matrices. 
The case corresponds to the following data generation model.
Let \(\mu_0 \) be the template object with mean \(r_0  \) and covariance  \(\Qv_0\). Let
\( \Qv_0 = U \Lambda^2_{0} U^{\T} \) be the eigenvalue decomposition of \(\Qv_0 \). 
Then all \(A_{\omega} \)'s defined in \eqref{def:admiss_transf} should be of the form \(A_{\omega} = U \diag(\alpha_{\omega})^2 U^{\T} \),
with \(\alpha_{\omega} \in \R^d \).

Consider two measures \(\nu_1\) and \(\nu_2 \) that belong to \(\mathcal{F}^{C}(\mu_0) \). As usual, denote
their first and second moments as \( (m_1, S_1) \) and \( (m_2, S_2) \) correspondingly.
Let eigenvalue decomposition of \(S_1 \) and \(S_2 \) be
\[
S_1 = U \diag(\lambda_1)^2 U^{\T}, \qquad S_2 = U \diag(\lambda_2)^2 U^{\T}.
\]
Then Wasserstein distance \eqref{def:wdist_gaussian} converts into
\begin{equation}
\label{def:W2_distance_commute}
W^2_2(\nu_1, \nu_2) = \| m_1 - m_2 \|^2 + \|\lambda_1  - \lambda_2 \|^2.
\end{equation}
Furthermore, the barycenter \(\mu_n \) \eqref{def:population_barycenter_general} converts into a measure, that
is characterized by  first and second moments \( (r_n, U \diag(\lambda_n)^2 U^{\T} )  \), where
\[
(r_n, \lambda_n) = \argmin_{(m,~\lambda)} \frac{1}{n}\sum_{i = 1} \bigl(\|m - m_i \|^2 + \| \lambda - \lambda_i \|^2 \bigr) .
\]
Thus, its first and second moments \( (r_n, \Qv_n) \) are
\begin{equation}
\label{def:emp_b_commute}
r_n = \frac{1}{n}\sum^n_{i = 1}m_i, \quad \Qv_n = U \diag\Biggl( \frac{1}{n}\sum^n_{i = 1} \lambda_i \Biggr)^2 U^{\T}.
\end{equation}
By analogy, empirical barycenter in the bootstrap world \( \mu^{\b}_n\) \eqref{def:barycenter_boot} is characterized by
\begin{equation}
\label{def:emp_b_commute_boot}
r^{\sbt}_n = \frac{1}{\sum_{i = 1}^n w_i}\sum_{i = 1}^n m_iw_i, \quad  \Qv^{\sbt}_n = U \diag\Biggl( \frac{1}{\sum_{i = 1}^n w_i}\sum^n_{i = 1} \lambda_i w_i \Biggr)^2 U^{\T}.
\end{equation}
Then the test statistic \(T_n\)
\begin{equation}
\label{def:stat_conf_set}
T_n = \sqrt{n\Bigl\| \frac{1}{n}\sum_{i} m_i - r_0  \Bigr \|^2 + n \Bigl\|\frac{1}{n}\sum_{i} \lambda_i - \lambda_0  \Bigr\|^2},
\end{equation}
and its counterpart in the bootstrap world \(T^{\sbt}_n\)
\begin{equation}
\label{def:stat_conf_set_boot}
T^{\sbt}_n = \sqrt{n\Bigl\| \frac{1}{\sum w_i}\sum_{i} m_iw_i -  \frac{1}{n}\sum_{i} m_i \Bigr \|^2 + n \Bigl\|\frac{1}{\sum w_i}\sum_{i} \lambda_i w_i - \frac{1}{n}\sum_{i} \lambda_i  \Bigr\|^2}. 
\end{equation}

\begin{theorem}[Bootstrap validity for confidence sets]
	\label{theorem:bootstrap_validity_final_barycenters}
	Let observed data \( \{\nu_1,..., \nu_n\} \) be an iid sample from model \(\mathcal{F}^{C}(\mu_0) \).
	 Under conditions from Section~\ref{section:conditions_barycenters} it holds
	  with probability \(\P^{\sbt} \geq 1- (e^{-\xx} + 2e^{-6\xx^2})\), \(\P \geq 1 - 6e^{-\xx} \)
	\[
	\bigl| \P\bigl(  T_n \leq \zz^{\sbt}_{n}(\alpha) \bigr) - \alpha \bigr| \leq \Delta_{\text{total}}(n),
	\]
	where \(\zz^{\sbt}_{n}(\alpha) \) comes from~\eqref{def:boot_quantile_barycenters} and
	\[
	\Delta_{\text{total}}(n) \leq \CONST/\sqrt{n},
	\]
	with \(\CONST \) is some generic constant.  
\end{theorem}
The proof is presented in Section~\ref{subsection:boot_validity}.

\section{Application to change point detection }
\label{section:change_point_wasserstein}
In this section we consider the procedure of change point detection in the 
flow of random measures. We assume, that observations are randomly sampled from some
scatter-location family with the unknown template object \(\mu_0 \). Change point occurs
if \(\mu_0 \) switches to another unknown template \(\mu_1 \). 
Namely, we consider two location-scatter families \( (\mathcal{F}_0) \) and \( (\mathcal{F}_1) \). 
The goal is to detect the switch as soon as possible.

\begin{description}
	\item[\( (\mathcal{F}_0) \)]
\emph{Let \(\mu_0 \) be a template object, s.t. \(X \sim \mu_0 \), with mean \(r_0 \) and covariance \(\Qv_0 \),}
\[
\mathcal{F}(\mu_0) = \bigl\{\nu_{\omega},~\text{s.t}~~ Y_{\omega}\sim \nu_{\omega} ~\text{and}~~ Y_{\omega} = A_{\omega}X + a_{\omega}, (A_{\omega}, a_{\omega}) \sim \P_0 \bigr\}.
\]
\end{description}

\begin{description}
	\item[\( (\mathcal{F}_1) \)]
	\emph{Let \(\mu_1 \) be a template object, s.t. \(X \sim \mu_1 \), with mean \(r_1 \) and covariance \(\Qv_1 \),}
	\[
	\mathcal{F}(\mu_1) = \bigl\{\nu_{\omega},~\text{s.t}~~ Y_{\omega}\sim \nu_{\omega} ~\text{and}~~ Y_{\omega} = B_{\omega}X + b_{\omega}, (B_{\omega}, b_{\omega}) \sim \P_1 \bigr\}.
	\]
\end{description}
Thus, observations that come from the data flow \(\{\nu_t\}_{t \in T} \) in hand are random deformations of either \(\mu_0 \) or \(\mu_1 \).
Let the time-moment \(t \) be fixed and consider it as a candidate to change point. Define as \(h\) the size of scrolling window.
The model is written as follows
\[
\begin{cases}
\nu_i \in (\mathcal{F}_0), & i \in \{t - h, ..., t - 1 \},\\
\nu_i \in  (\mathcal{F}_1), & i \in \{t, ..., t + h - 1 \}.
\end{cases} 
\]
Now we have to choose between following alternatives:
\[
H_0 : (\mathcal{F}_0) = (\mathcal{F}_1), \quad H_1: (\mathcal{F}_0) \neq (\mathcal{F}_1).
\]
Further we present a non-parametric testing procedure. As previously, each observed measure \(\nu_i \) is characterized by 
its mean \(m_i \) and covariance matrix \(S_i \), that are known.
The idea of the method is to compare how close to each other template estimators \(\mu_{l}(t)\) and \(\mu_{r}(t) \),
that are computed using data in the left and right halves of the scrolling window respectively. 
These estimators are defined as
	\begin{equation}
	\label{def:mu_left}
	\mu_{l}(t) \eqdef \argmin_{\mu \in \mathcal{P}_2(\R^d)} \frac{1}{h}\sum_{i = t - h}^{t - 1}W^2_2(\mu, \nu_i),
	\end{equation}
	\begin{equation}
	\label{def:mu_right}
	\mu_{r}(t) \eqdef \argmin_{\mu \in \mathcal{P}_2(\R^d)} \frac{1}{h}\sum_{i = t }^{t + h - 1}W^2_2(\mu, \nu_i).
	\end{equation}
Then the test statistic is written as
	\begin{equation}
	\label{def:cp_test_barycenters}
	T_h(t) = \sqrt{h}W_2\bigl( \mu_{l}(t), \mu_{r}(t) \bigr).
	\end{equation}
	Change point is supposed to be detected at the moment \(t \) if the test exceeds some critical level \(\zz_h(\alpha) \):
	\[
	T_h(t) \geq \zz_h(\alpha), 
	\]
	where \(\zz_h(\alpha) \) is defined as \(\alpha\)-quantile of \(T_h(t) \) under \(H_0 \):
	\[
	\zz_h(\alpha) \eqdef \argmin_{\zz}\bigl\{\P_0\bigl(T_h(t) \geq \zz\bigr) \leq \alpha \bigr\}.
	\]
	As soon as \(\zz_h(\alpha) \) can not be computed analytically, the core idea of the approach 
	is to replace it with bootstrapped counterpart \( \zz^{\sbt}_h(\alpha) \).
	
	\subsection{Bootstrap procedure}
		While tuning critical values, we assume, that an observed training sample \(\{\nu_1,... \nu_M\}\)
		is iid and belongs to \((\mathcal{F}_0)  \).
		Following the already presented framework, we define counterparts of~\eqref{def:mu_left} and~\eqref{def:mu_right} in the bootstrap world:
			\begin{equation}
			\label{def:mu_left_bootstrap}
			\mu^{\sbt}_{l}(t) \eqdef \argmin_{\mu \in \mathcal{P}_2(\R^d)} \frac{1}{\sum_{i}w_i}\sum_{i = t - h}^{t - 1}W^2_2(\mu, \nu_i)w_i,
			\end{equation}
			\begin{equation}
			\label{def:mu_right_bootstrap}
			\mu^{\sbt}_{r}(t) \eqdef \argmin_{\mu \in \mathcal{P}_2(\R^d)} \frac{1}{\sum_{i}w_i} \sum_{i = t }^{t + h - 1}W^2_2(\mu, \nu_i)w_i,
			\end{equation}
			where weights \(w_i\) follow Condition \( {(E\!D^{W}_{\b})} \) and are independent of the observed data set\\
			 \(\{\nu_{t-h},..., \nu_{t + h - 1} \} \). The bootstrapped statistic test is
				\begin{equation}
				\label{def:cp_test_boot_barycenters}
				T^{\sbt}_h(t) = \sqrt{h}W_2\bigl( \mu^{\sbt}_{l}(t), \mu^{\sbt}_{r}(t) \bigr).
				\end{equation}
			We define bootstrapped quantile \(\zz^{\sbt}_{h}(\alpha) \) as
			\begin{equation}
			\label{def:boot_qunatile_cp_barycenters}
				\zz^{\sbt}_{h}(\alpha) \eqdef \argmin_{\zz}\bigl\{\P^{\sbt}\bigl(T^{\sbt}_h(t) \geq \zz\bigr) \leq \alpha \bigr\}.
			\end{equation}
			The procedure of critical value calibration is presented in Algorithm~\ref{alg:computation_z_change_point}.

\subsection{The case of commuting matrices}
As previously (see Section~\ref{section:confidence_sets}) we consider the setting where covariance matrices commute. 
In this case scatter-location families \((\mathcal{F}_0)\) and \((\mathcal{F}_1)\) 
turn into \((\mathcal{F}^{C}_0)\) and \((\mathcal{F}^{C}_1)\).
\begin{description}
	\item[\( (\mathcal{F}^{C}_0) \)]
	\emph{Let \(\mu_0 \in \mathcal{P}^{ac}_{2}(\R^d) \) be a template object, 
		s.t. \(X \sim \mu_0 \), with mean \(r_0 \) and covariance \(\Qv_0 \),
		The eigenvalue decomposition of \(\Qv_0 = U \Lambda^2_0 U^{\T} \)}
	\[
	\mathcal{F}^{C}(\mu_0) = \bigl\{\nu_{\omega},~\text{s.t}~~ 
	Y_{\omega}\sim \nu_{\omega} ~\text{and}~~ 
	Y_{\omega} = A_{\omega}X + a_{\omega}, 
	(A_{\omega}, a_{\omega}) \sim \P_0 \bigr\}.
	\]
	\emph{
		Furthermore, each \(A_i\) is decomposed as \(A_i = U \alpha_i U^{\T} \), 
		with \(\alpha_i \) diagonal positively defined matrix.
		We assume \( (a_i, \alpha_i) \sim \P_0\) , \(\P_0 \) is defined on \(\R^d\times\mathbb{S}_{+}(d, \R) \),
		and \(a_i \) and \(\alpha_i \) are independent of each other.
		}
\end{description}

\begin{description}
	\item[\( (\mathcal{F}^{C}_1)\)]
	\emph{Let \(\mu_1 \in \mathcal{P}^{ac}_{2}(\R^d) \) be a template object, 
	s.t. \(X \sim \mu_1 \), with mean \(r_1 \) and covariance \(\Qv_1 \),
	The eigenvalue decomposition of \(\Qv_1 = U \Lambda^2_1 U^{\T} \)}
	\[
	\mathcal{F}^{C}(\mu_1) = \bigl\{\nu_{\omega},~\text{s.t}~~ 
	Y_{\omega}\sim \nu_{\omega} ~\text{and}~~ 
	Y_{\omega} = B_{\omega}X + b_{\omega}, 
	(B_{\omega}, b_{\omega}) \sim \P_1 \bigr\}.
	\]
	\emph{Furthermore, each \(B_i\) is decomposed as \(B_i = U \beta_i U^{\T} \), 
		with \(\beta_i \) diagonal positively defined matrix.
		We assume \( (b_i, \beta_i) \sim \P_1\) and \(\P_1 \) is defined on \(\R^d\times\mathbb{S}_{+}(d, \R) \),
			and \(b_i \) and \(\beta_i \) are independent of each other.
		}
\end{description} 
Each observed measure \(\nu_i \) is characterized by mean \(m_i \) and covariance matrix \(S_i \),
the latter one is decomposed as
\(S_i = U \Lambda^2_i U^{\T}\), ~with \(\Lambda^2_i = \diag(\lambda_i)^2 \).
Taking into account~\eqref{def:W2_distance_commute},~\eqref{def:emp_b_commute} 
and~\eqref{def:emp_b_commute_boot} one can see, 
that the squared test statistic~\eqref{def:cp_test_barycenters} 
and its reweighted counterpart~\eqref{def:cp_test_boot_barycenters} are
\begin{equation}
\label{def:cp_test_commute}
T^{2}_h(t) = h\Biggl\| \frac{1}{h}\sum_{i \in \Lt} m_i -  \frac{1}{h}\sum_{i \in \Rt} m_i \Biggr \|^2 + h \Biggl\|\frac{1}{h}\sum_{i \in \Lt} \lambda_i - \frac{1}{h}\sum_{i \in \Rt} \lambda_i  \Biggr\|^2,
\end{equation}
\begin{align}
\label{def:cp_test_commute_boot}
T^{\sbt 2}_h(t) = &h\Biggl\| \frac{1}{\sum_{i \in \Lt} w_i}\sum_{i \in \Lt} m_iw_i -  \frac{1}{\sum_{i \in \Rt} w_i}\sum_{i \in \Rt} m_iw_i \Biggr \|^2 \nonumber \\
&+ h \Biggl\|\frac{1}{\sum_{i \in \Lt} w_i}\sum_{i \in \Lt} \lambda_i w_i - \frac{1}{\sum_{i \in \Rt} w_i}\sum_{i \in \Rt} \lambda_iw_i  \Biggr\|^2, 
\end{align}
where
\[
\Lt \eqdef \{t-h,..., t-1 \}, \quad \Rt \eqdef \{t,..., t + h -1 \}.
\]
The next theorem shows, that bootstrapped quantile \(\zz^{\sbt}_{h}(\alpha) \)~\eqref{def:boot_qunatile_cp_barycenters} 
can be used instead of \(\zz_{h}(\alpha) \).
\begin{remark}
	For transparency of further presentation, 
	here we use non-intersected running windows. 
	Thus, for each \( t\) observed \(T_h(t) \) are iid.
	Thus, there is a slight modification in Algorithm~\ref{alg:computation_z_change_point}.
	Instead of considering each \(t \in \{h + 1,..., M - h \} \),
	we consider \(t \in \mathcal{I} = \{h + 1, 3h + 1,..., M - h \} \).
\end{remark}
\begin{theorem}[Bootstrap validity for change point detection]
	\label{theorem:bootstrap_validity_final_cp_barycenters}
	Let observed data \( \nu_1,..., \nu_M \) be iid sample from model \(\mathcal{F}^{C}(\mu_0) \).
	Let the size of running window \(h \) be fixed.
	Under conditions from Section~\ref{section:conditions_barycenters} it holds	with probability\( \P \geq 1 - 6e^{-\xx} - 2e^{-6\xx^2} \), 
	\(\P^{\sbt} \geq 1 -( e^{-\xx} + 2e^{-6\xx^2}) \)
	\[
	\bigl| \P\bigl(  \max_{t \in \mathcal{I}}T_h(t) \leq \zz^{\sbt}_{ h}(\alpha) \bigr) - \alpha \bigr| \leq \tfrac{M}{2h}\Delta_{\text{total}, cp}(h),
	\]
	where \(\zz^{\sbt}_{h}(\alpha) \) comes from~\eqref{def:boot_quantile_barycenters} and
	\[
	\Delta_{\text{total}, cp}(h) \leq \CONST/\sqrt{h},
	\]
	here \(\CONST \) is some generic constant and \(\Delta_{\text{total, cp}}(h)\) is defined in~\eqref{def:error_total_cp_barycenters} .  
\end{theorem}

\newpage
\section{Algorithms}
\label{section:algo_barycenters}
\begin{algorithm}[!h]
	\KwData{Sample \(\{\nu_1,... \nu_n\}\), distribution of \(w_i\), false-alarm rate \(\alpha \) }
	\KwResult{\(\zz_{n}^\b(\alpha)\)}
	initialize the number of iterations \(M\)\\
	\For{$ i  \in \{1,...J \}$}{ 
		sample an iid set \(\{w_1,..., w_n\}  \)\\
		compute \(\mu^{\b}_n(i) \) using~\eqref{def:barycenter_boot}\\
	}
	for each \(t \geq 0 \) compute ecdf \(F^\b_{n, M}(t) \):
	\[
	F^\b_{n, M}(t) = \frac{1}{M}\sum_{i = 1}^n\I\Bigl\{ \sqrt{n}W_2\bigl( \mu^{\b}_n(i), \mu_n \bigr) \leq t \Bigr\} 
	\]
	compute bootstrapped quantile \(\zz_{n}^\b(\alpha)\):
	\[
	\zz_{n}^\b(\alpha) = \inf_{t \geq 0} \Bigl\{F^\b_{n, M}(t) \geq 1 - \alpha \Bigr\}
	\]
	
	\caption{Computation of critical value \(	\zz_{n}^\b(\alpha)\)}	\label{alg:computation_z}
\end{algorithm}

\begin{algorithm}[!h]
	\KwData{Sample \(\{\nu_1,... \nu_M\}\), window width \(h\), distribution of \(w_i\), false-alarm rate \(\alpha \) }
	\KwResult{\(\zz^{\sbt}_{h}(\alpha)\)}
	initialize the number of iterations \(J\)\\
	initialize  vector \(R \), length(\(R\)) = \(J\)\\
	\For{$ i  \in \{1,... J\}$}{ 
		sample an iid set \(\{w_1,..., w_M\}  \)\\
		\(R(j) := 0\);\\
		\For{$t \in \{h+1, ..., M-h \} $}{
			compute \(\mub^{\b}_{l}(t, i) \), \(\mub^{\b}_{r}(t, i) \), \(T^{\sbt}(t, i) \)   using~\eqref{def:mu_left_bootstrap},~\eqref{def:mu_right_bootstrap} and~\eqref{def:cp_test_boot_barycenters}\\
			\(R(j) := \max\bigl( R(j), T^{\sbt}(t, i)\bigr)\)\\
		}
		
	}
	for \(\xx \geq 0 \) compute ecdf \(F^{\sbt}_{h}(\xx) \):
	\vspace{-5pt}
	\[
	F^\b_{h}(\xx) = \frac{1}{J}\sum_{j = 1}^J\I\Bigl\{ R(j) \leq \xx \Bigr\} 
	\]
	compute bootstrapped quantile:
\(
	\zz^{\sbt}_{h}(\alpha) = \inf_{\xx \geq 0} \Bigl\{F^\b_{h}(\xx) \geq 1 - \alpha \Bigr\}
	\)
	
	\caption{Computation of critical value \(	\zz^{\sbt}_{h}(\alpha)\)}
	\label{alg:computation_z_change_point}
\end{algorithm}
\section{Experiments}
\label{section:experiments}
In this section we consider experiments on real and artificial data. 
The examples give intuition that the method of
construction of confidence sets and change point detection
is valid in more general cases, than considered in
Theorem~\ref{theorem:bootstrap_validity_final_barycenters} and 
Theorem~\ref{theorem:bootstrap_validity_final_cp_barycenters}.
\subsection{Coverage probability of the true object \(\mu^* \)}
This section examines the quality of the approximation 
of \(\zz_{n}(\alpha)\) by \( \zz^{\b}_{n}(\alpha) \) in case of confidence sets construction. 
We assume that each observed object \(\nu \sim \P \) is a random transformation
of a template \(\mu^*_0 \), s.t.
\[
\E_{\tiny \P} \nu = \mu^*_0.
\]
As before, denote as \(\mu_n\) the empirical barycenter of some i.i.d. set  \(\{\nu_1,..., \nu_n \} \), \( \nu_i \sim \P \). 
For a fixed false-alarm rate \(\alpha \) the confidence set \(\mathcal{C}\bigl(\zz_{n}(\alpha) \bigr)\) 
is defined in~\eqref{def:conf_sets_barycenters}.
Our goal is to check the closeness of \(\zz^{\b}_{n}(\alpha) \) computed with Algorithm~\ref{alg:computation_z} 
and \(\zz^{\b}_{n}(\alpha)\).
To do that let's introduce some other template object \(\mu^*_1 \): \(\mu^*_1 \neq \mu^*_0 \). 
We are interested in the estimation of the following probabilities:
\begin{equation}
\label{def:false_alarm_rate}
\P\Bigl( \mu^*_0 \in \mathcal{C}\bigl(\zz^{\b}_{n}(\alpha)  \bigr) \Bigr), \qquad \P\Bigl( \mu^*_1 \not\in \mathcal{C}\bigl(\zz^{\b}_{n}(\alpha)  \bigr) \Bigr).
\end{equation}

To do that we follow the work by \citet{cuturi2013} and consider as a template object \(\mu^*_0 \) 
two concentric circles, that are depicted at the left-hand side of Fig.~\ref{fig:recovery}.
The middle panel contains four samples of \(\nu \). Each \(\nu_i \) is 
obtained by random shifts and dilations of each circle.
The last box depicts the barycenter \(\mu_n \).
Naturally, each image can be considered as a uniform measures on \( \R^2\). 
As \(\mu^*_1 \) we consider a single shifted random ellipse presented at Fig.\ref{fig:ellips}. 
To compute optimal transport we use iterative Bregman 
projections algorithm presented in~\citet{benamou2015iterative}.
The code can be found on Git-Hub repository following the link 
\href{https://github.com/gpeyre/2014-SISC-BregmanOT}{https://github.com/gpeyre/2014-SISC-BregmanOT}.
\begin{remark}
It is worth noting that the algorithm solves a penalized problem rather then the original one. 
In other words, instead of minimizing
\[
 \int_{\R^d} \| x- y\|^2 d\pi(x,y) \rightarrow \min_{\pi \in \Pi(\mu, \nu)},
\]
it optimizes the following target function
\[
\int_{\R^d} \| x- y\|^2 d\pi(x,y) + \gamma \int_{\R^d}\pi(x, y)d\pi(x, y) \rightarrow \min_{\pi \in \Pi(\mu, \nu)},
\]
The solution of regularized problem converges to the solution of the original one
with the decay of \(\gamma\). Thus, choosing relatively small regularization parameter \(\gamma =.005  \),
we assume that the obtained results are quite close to the solution
of the original problem.
\end{remark}

\begin{figure}[ht!]
	\centering
	\includegraphics[width=.85\textwidth]{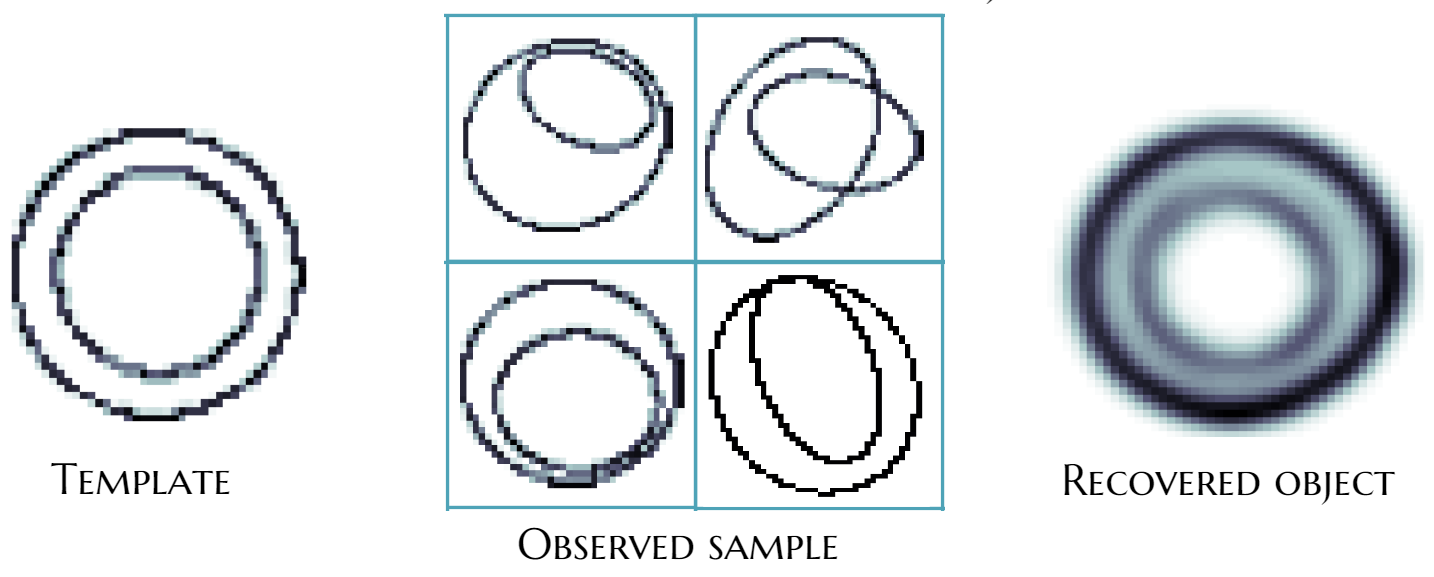}
	\caption{Recovery of the template object} 
	 \label{fig:recovery}
\end{figure}
\begin{figure}[ht!]
	\centering
	\includegraphics[width=.85\textwidth]{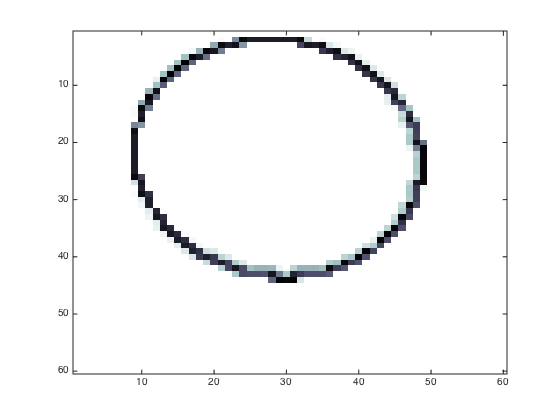}
	\caption{\(\mu^*_1 \) is a dilated and shifted single ellipse} 
	\label{fig:ellips}
\end{figure}
The experiments are carried for eight sample sizes \(n \in \{5, 10, 15, 20, 25, 30, 35, 40 \} \).
Confidence intervals are estimated for two different false-alarm rates \(\alpha \in \{.05, .01 \} \).
Bootstrap weights follow Poisson distribution \(w_i \sim \text{Po}(1) \).
Empirical estimators of~\eqref{def:false_alarm_rate} are presented 
in Tables~\ref{table:true} and~\ref{table:test}.

\begin{table}[h!]
\caption{\(\hat{\P}\Bigl( \mu^*_0 \in \mathcal{C}\bigl(\zz^{\b}_{n}(\alpha)  \bigr)  \Bigr)\)}
\centering
\begin{tabular}{c c c c c c c c c }
\hline\hline
                                 & \(n = 5 \)  & \(n = 10 \) &\(n = 15 \) & \(n = 20 \) & \(n = 25 \) & \(n = 30 \) & \(n = 35 \) & \(n = 40 \)\\ [0.5ex] 
\hline
 \(\alpha = .05 \) \vline &      \(.90 \)        &\(.91\)        &       \(.97\) & \(.97 \) & \(.98\) & \(.92\) & \(.97\) & \(.93 \)\\
\(\alpha = .01 \) \vline  &       \(.91\)        &\(.96\)     &\(.99\) & \(.99\) & \(1\) & \(.98 \) & \(.99\) & \( .96\)\\
\hline
\end{tabular}
\label{table:true}
\end{table}

\begin{table}[h!]
	\caption{\(\hat{\P}\Bigl( \mu^*_1 \not\in \mathcal{C}\bigl(\zz^{\b}_{n}(\alpha)  \bigr)  \Bigr)\)}
	\centering
	\begin{tabular}{c c c c c c c c c}
		\hline\hline
		& \(n = 5 \)  & \(n = 10 \) &\(n = 15 \) & \(n = 20 \) & \(n = 25\) & \(n = 30 \) & \(n = 35\) & \(n = 40 \) \\ [0.5ex] 
		\hline
		\(\alpha = .05 \) \vline &       \(.85\)         &\(.84\)        &       \(.94\) & \(.94 \) & \(.94\) & \(.98\) &\(.98\) & \(.97\)\\
		\(\alpha = .01 \) \vline  &       \(.75\)       &\(.73\)     &\(.76\) & \(.84\)  & \(.86\) &  \(.91\) &\(.90\) & \(.94\) \\
		\hline
	\end{tabular}
	\label{table:test}
\end{table}

\subsection{Experiments on the real data}
This section presents algorithm performance of Algorithm~\ref{alg:computation_z} on the real data. 
We use MNIST (handwritten digit database) 
\href{http://yann.lecun.com/exdb/mnist/}{http://yann.lecun.com/exdb/mnist/}. 
It contains around 60000 indexed black-and-white images. 
Each image is a bounding box of \(28 \times 28\) pixels with a written digit inside. 
Several examples are presented at Fig.~\ref{fig:mnist}. 
All symbols are approximately of the same size.  
Fig.~\ref{fig:mnist_barycenter} presents empirical barycenter for each digit,
computed using all images in the database.

Naturally, there is no "template object" for hand-written symbols. 
Thus, the predefined template \(\mu_0 \) 
can be replaced with population barycenter \(\mu^*\).
However, on practice it can not be calculated as well.
To carry out the test, instead of \(\mu^* \) we use
an empirical barycenter \(\mu_N \) of a large 
homogeneous random sample \(\{\nu_1,..., \nu_N \} \).
\begin{figure}[ht!]
	\centering
	\includegraphics[width=.8\textwidth]{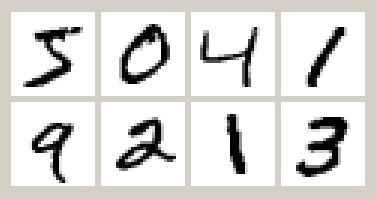}
	\caption{Random sample from MNIST database} 
	 \label{fig:mnist}
\end{figure}

\begin{figure}[ht!]
	\centering
	\includegraphics[width=.95\textwidth]{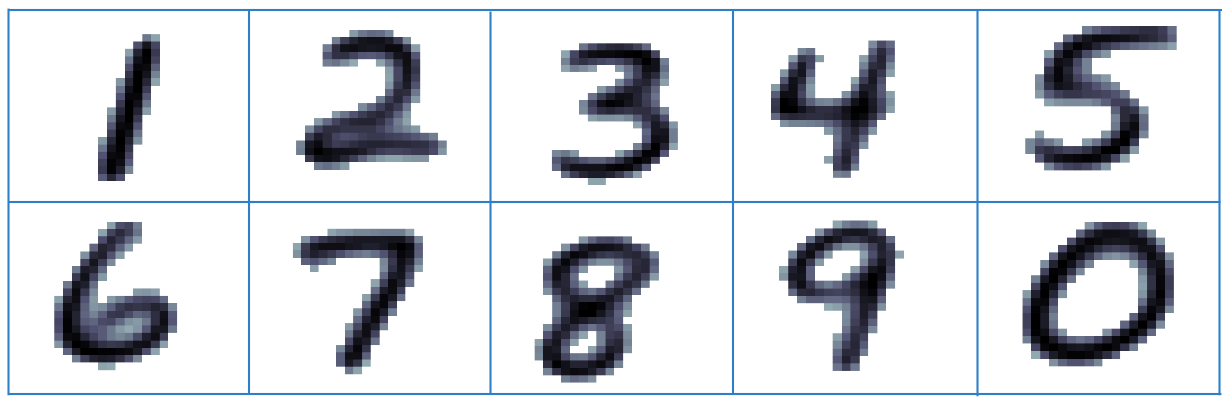}
	\caption{Barycentric digits, computed over the whole MNIST database} 
	\label{fig:mnist_barycenter}
\end{figure}

Now we briefly explain the experiment setting.
Denote as \(\mathcal{S}\) the set of all MNIST images. 
As before, each image \(\nu \in \mathcal{S}\) can be considered as some measure on \(\R^2\) 
with a finite support of size \(28\times 28\). 
First fix some reference and test digits and denote them as \(r\) and \(t\) respectively.
Then extract all \(r\)- and \(t\)-entries from \(\mathcal{S}\)
\[
\mathcal{S}_r = \{\text{set of all \(r\)'s}  \},
\qquad
\mathcal{S}_t = \{\text{set of all \(t\)'s}  \}.
\]
From the reference set \( \mathcal{S}_r \) we then sample some 
\( \{\nu^{r}_1,..., \nu^{r}_n\} \subset \mathcal{S}_r\) 
and denote as \(\mu^r_n\) its empirical barycenter. 
Let \(\mathcal{C}_n(\zz^{\sbt}_{\alpha}) \) be \(\alpha\)-confidence set around  \(\mu^r_n\).
The test procedure aims to estimate two following probabilities
\[
{\P}\Bigl( \mu^{r} \in \mathcal{C}\bigl(\zz^{\b}_{n}(\alpha)  \bigr) \Bigr), \qquad{\P}\Bigl( \mu^{t} \not\in \mathcal{C}\bigl(\zz^{\b}_{n}(\alpha)  \bigr) \Bigr),
\]
where \(\mu^{r}\) and \(\mu^{t}\) are empirical barycenters,
computed using whole sets \( \mathcal{S}_r \) and \(\mathcal{S}_t  \) respectively (see Fig.~\ref{fig:mnist_barycenter}). 
As the reference and test digits are used \(r = 3\) and \(t = 8\) respectively.
We consider eight sample sizes \(n \in\{5, 10, 15, 20, 25, 30, 35, 40\} \) 
and two confidence levels \(\alpha \in \{.05, .01 \} \).
Empirical probabilities are estimated using \(100\) experiments for each \(n\) and \(\alpha\).
The results are presented in Tables~\ref{table:true_digits} and~\ref{table:test_digits}.

\begin{table}[h!]
	\caption{\(\hat{\P}\Bigl( \mu^{3} \in \mathcal{C}\bigl(\zz^{\b}_{n}(\alpha)  \bigr)  \Bigr)\)}
	\centering
	\begin{tabular}{c c c c c c c c c }
		\hline\hline
		                                 &\(n = 5 \)& \(n = 10 \)&\(n = 15 \)&\(n = 20 \)& \(n = 25 \)& \(n = 30 \) & \(n = 35 \) & \(n = 40 \)\\ [0.5ex] 
		\hline
		\(\alpha = .05 \) \vline  &   \(.80 \)&\(.90\)       &     \(.92\) & \(.95 \)    & \(.95\) & \(.99\) & \(.99\) &  \(.99\)\\
		\(\alpha = .01 \) \vline  &   \(.91\) &\(.94\)       &    \(.97\)  & \(.97\)     & \(.95\) & \(.99\) & \(.99\) &  \(1\)\\
		\hline
	\end{tabular}
	\label{table:true_digits}
\end{table}

\begin{table}[h!]
	\caption{\(\hat{\P}\Bigl( \mu^{8} \not\in \mathcal{C}\bigl(\zz^{\b}_{n}(\alpha)  \bigr)  \Bigr)\)}
	\centering
	\begin{tabular}{c c c c c c c c c}
		\hline\hline
		& \(n = 5 \)  & \(n = 10 \) &\(n = 15 \) & \(n = 20 \) & \(n = 25\) & \(n = 30 \) & \(n = 35\) & \(n = 40 \) \\ [0.5ex] 
		\hline
		\(\alpha = .05 \) \vline &       \(.92\)         &\(1\)        &       \(1\) & \(1 \) & \(1\) & \(1\) &\(1\) & \(1\)\\
		\(\alpha = .01 \) \vline  &       \(.80\)       &\(.96\)     &\(.99\) & \(1\)  & \(1\) &  \(1\) &\(1\) & \(1\)\\
		\hline
	\end{tabular}
	\label{table:test_digits}
\end{table}

Fig.~\ref{fig:stat} shows an empirical distribution of \(2\)-Wasserstein distance in two following cases: 
\(\sqrt{n}W_2(\mu^{r}_n, \mu^{r})\) (the left histogram) and \(\sqrt{n}W_2(\mu^{r}_n, \mu^{t})\) (the right histogram) respectively. 
The red vertical line is \(\alpha \)-quantile \(\zz^{\b}_{n}(\alpha)\) computed with the bootstrap procedure.
Fixed parameters are  \(n = 20 \), \(\alpha = .01 \), \(r = 3\), \(t = 1\).
\begin{figure}[ht!]
	\centering
	\includegraphics[width = .95\textwidth]{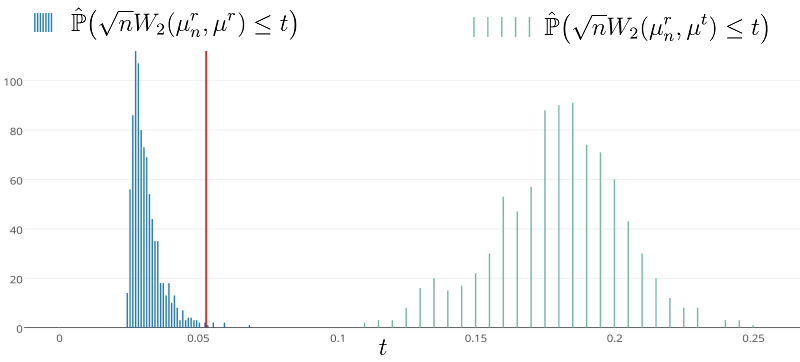}
	\caption{Distribution of \( \sqrt{n}W_2(\mu^r_n, \mu^r)\) and \( \sqrt{n}W_2(\mu^r_n, \mu^{t})\)} 
	\label{fig:stat}
\end{figure} 

\subsection{Application to change point detection}
This section illustrates the performance of change point detection procedure.
We consider the following data generation process.
Before the change point observed data is generated from some template \(\mu^*_0 \) and from \(\mu^*_1 \) afterwords.
Two examples are presented at Fig.~\ref{fig:curv_triang}. The upper panel shows a switch from nested ellipses
to curved triangles. The bottom panel refers to switch from handwritten digits "three" to "five". The second data set
is randomly sampled from MNIST databases.
 Let \(T_h(t) \) be the test statistics 
\[
T_h(t) = \sqrt{h}W_2\bigl(\mu^l_{h}(t), \mu^r_{h}(t)\bigr),
\]
where \(\mu^l_{h}(t) \) and \(\mu^r_{h}(t) \) stand for the barycenter of \(\{\nu_{t - h},..., \nu_{t - 1}  \} \) and \(\{\nu_{t},..., \nu_{t + h - 1}  \} \) respectively.
\begin{figure}[h!]
	\centering
	\includegraphics[width=.85\textwidth]{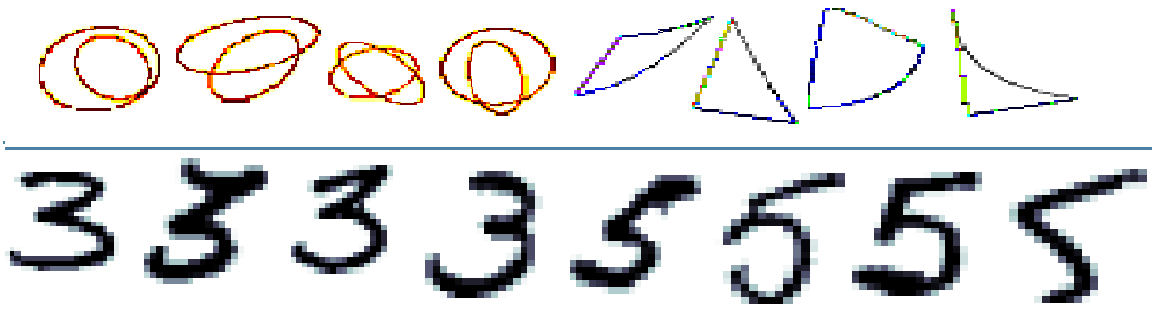}
	\caption{Change point in a flow of random images} 
	\label{fig:curv_triang}
\end{figure} 

Fig.~\ref{fig:cp_detection} illustrates work of the algorithm on artificial data, namely on the stream
that switches from random nested ellipses to curved triangles. Critical value \(\zz^{\sbt}_h(\alpha) \) for \(\alpha = .01 \) 
is computed with Algorithm~\ref{alg:computation_z_change_point} and depicted with horizontal line.
We use the width of scrolling window \(2h = 10 \). 
Switch occurs at \(t = 40\) and is marked with the black vertical line. 
The panel below shows the data before and after the change point, namely 
\(\nu_t \) for \(t \in \{36, 37, 38, 39, 40, 41, 42, 43 \} \).
Fig.~\ref{fig:cp_detection_digits} shows the behaviour of \(T_h(t) \) on real data set. Switch occurs at
\(t = 200 \), the width of scrolling window is \(2h = 100 \). Horizontal lines, that refer to critical levels  \(\alpha = .01 \) and \(\alpha = .05 \) respectively 
are computed with Algorithm~\ref{alg:computation_z_change_point}.
As before, the panel below depicts observations in the vicinity of the change point: 
\(\nu_t \) for \(t \in \{196, 197, 198, 199, 200, 201, 202, 203 \} \).
\newpage
\begin{figure}[ht!]
	\centering
	\includegraphics[width=.95\textwidth]{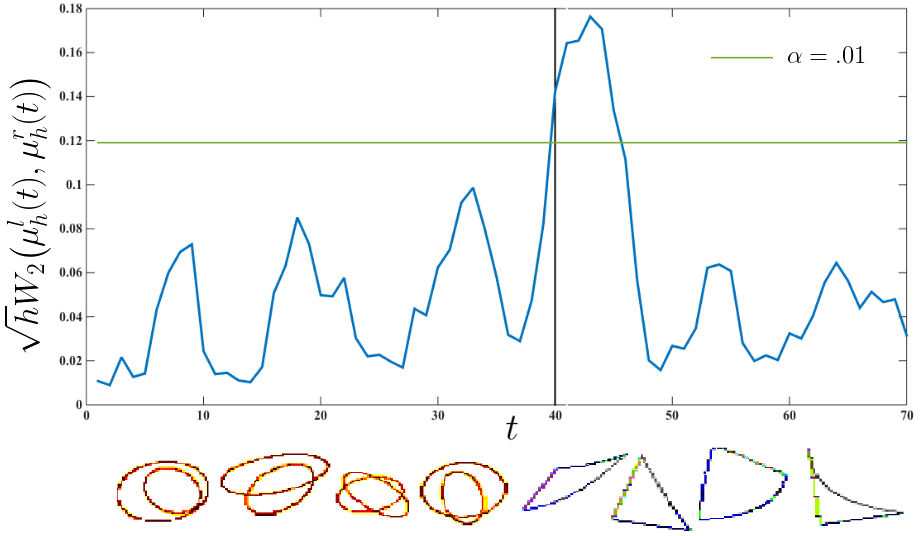}
	\caption{Distribution of \( \sqrt{h}W_2\bigl(\mu^l_h(t), \mu^r_h(t)\bigr)\)} 
	\label{fig:cp_detection}
\end{figure} 
\begin{figure}[ht!]
	\centering
	\includegraphics[width=.95\textwidth]{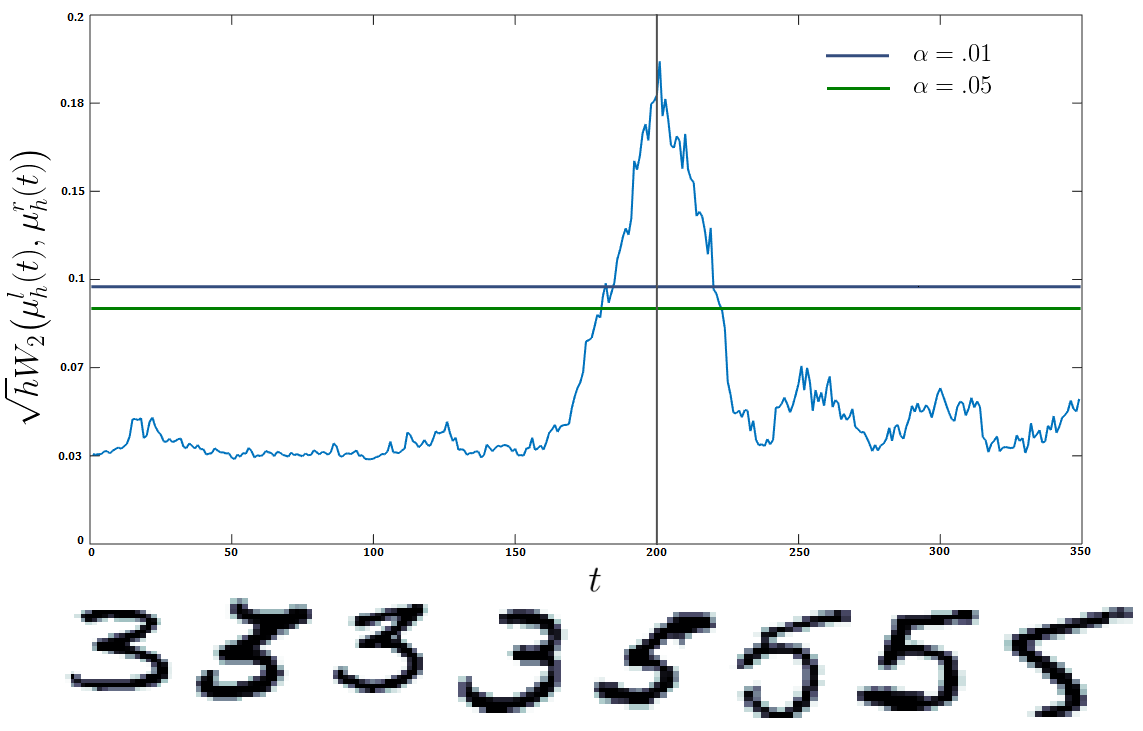}
	\caption{Distribution of \( \sqrt{h}W_2\bigl(\mu^l_h(t), \mu^r_h(t)\bigr)\)} 
	\label{fig:cp_detection_digits}
\end{figure}

\appendix
\section{Appendix}
\label{chapter:proofs_barycenters}
\subsection{Conditions}
\label{section:conditions_barycenters}
\begin{description}
	\item[\( (Sm^{W}_b) \)] 
\emph{Let \(X_i \), \(i = 1,.., n \) be a set of independent centred random variables defined on \(\R^d \).
 Let \(S^2 \eqdef  \sum_{i =1}^n \E X_{i}X^{\T}_{i} \). We assume that} 
\[
\|S^{-1}X_i \| \leq \CONST / \sqrt{n}.
\]
\end{description}

\begin{description}
	\item[\( (ED^{W}_b) \)] 
\emph{Let \(X\) be s.t.}
\[
\sup_{\gamma \in \R^d} \log \E \exp\bigl( \lambda \gamma^{\T}X \bigr) \leq \frac{\lambda^2\nu^2}{2}, \quad |\lambda| \leq \sigma,
\]
\emph{ where \(\nu \geq 1 \) and \(\sigma > 0 \) are some constants.}
\end{description}
\begin{description}
	\item[\( {(E\!D^{W}_{\b})} \)]
	\emph{For all \(i \) the set of bootstrap weights \(\{w_i\} \) have unit mean and variance}
	\[
	\E^{\b} w_i = 1, \quad \Var^{\b}(w_i) = 1.
	\]
	\emph{Furthermore,}
	\[
	\log \E \exp \bigl\{ \lambda {(w_i - 1)} \bigr\} \leq  \frac{\nu^2_0\lambda^2}{2}, \quad |\lambda| \leq {g}.
	\]
\end{description}

\subsection{Validity of the bootstrap procedure for confidence sets}
\label{subsection:boot_validity}
First note, that the test statistic~\eqref{def:stat_conf_set} can be written as
\[
T^2_n = \|{\xi}\|^2 + \|{\eta}\|^2,
\] 
where
\begin{equation}
\label{def:xi}
\xi_i \eqdef m_i- m^*, \quad {\xi} \eqdef \frac{1}{\sqrt{n}}\sum_{i}\xi_i,
\end{equation}
\begin{equation}
\label{def:eta}
\eta_i \eqdef \lambda_i- \lambda^*, \qquad {\eta} \eqdef \frac{1}{\sqrt{n}}\sum_{i}\eta_i.
\end{equation}
Note that all \(\xi_i \) and \( \eta_i \) are centred: \( \E\xi_i = 0, \quad  \E\eta_i = 0 \). 
We define their covariances as
\[
\Var(\xi_i) \eqdef \mathtt{v}^2_{\xi}, \qquad \Var(\eta_i) \eqdef \mathtt{v}^2_{\eta}.
\]
Similarly, \(T^{\b 2}_n \)~\eqref{def:stat_conf_set_boot} can be rewritten as
\[
T^{\b2}_n = \|{\xi}^{{\sbt}}\|^2 + \|{\eta}^{{\sbt}}\|^2,
\]
\begin{equation}
\label{def_xi_boot}
{\xi}^{{\sbt}} \eqdef \frac{\sqrt{n}}{\sum_{i} w_i}\sum_{i}\xi^{{\sbt}}_i = \frac{\sqrt{n}}{\sum w_i}\sum_{i} m_iw_i -  \frac{1}{\sqrt{n}}\sum_{i} m_i.
\end{equation}
\begin{equation}
\label{def_eta_boot}
{\eta}^{{\sbt}} \eqdef  \frac{\sqrt{n}}{\sum_{i} w_i}\sum_{i}\eta^{{\sbt}}_i w_i  =  \frac{\sqrt{n}}{\sum w_i}\sum_{i} \lambda_i w_i - \frac{1}{\sqrt{n}}\sum_{i} \lambda_i.
\end{equation}
where
\begin{equation}
\label{def_xi_i_boot}
\xi^{{\sbt}}_i \eqdef w_i \Bigl(\xi_i  -  \frac{1}{\sqrt{n}}\sum_{i}\xi_i \Bigr), \quad \eta^{{\sbt}}_i \eqdef w_i \Bigl(\eta_i  -  \frac{1}{\sqrt{n}}\sum_{i}\eta_i \Bigr),
\end{equation}

Before proving the main result (Theorem~\ref{theorem:bootstrap_validity_final_barycenters}), we have to prove an auxiliary lemma.
\begin{lemma}
	\label{lemma:klom_dist_conf_sets_barycenters}
Let \(\{\nu_1,..., \nu_n\} \) be an iid sample from family \(\mathcal{F}_0 \) 
and let conditions of Section~\ref{section:conditions_barycenters} 
be fulfilled, then for test statistic \(T_n \) and its counterpart in the bootstrap world \(T^{{\sbt}}_n \)
the following relation holds with probabilities \(\P^{{\sbt}} \geq 1- (e^{-\xx} + 2e^{-6\xx^2})\), \(\P \geq 1 - 6e^{-\xx} \)
\[
\sup_{\zz > 0}\Bigl| \P(T_n \leq \zz) - \P^{{\sbt}}(T^{{\sbt}}_n \leq \zz) \Bigr| \leq \Delta_{\text{total}},
\]
with \(\Delta_{\text{total}} \eqdef \Delta_{\xi, \text{total}} + \Delta_{\eta, \text{total}} \)  
where \(\Delta_{\xi, \text{total}} \), \(\Delta_{\eta, \text{total}} \) comes from Lemma~\ref{thm:indBSValidity}.
\end{lemma}
Informally speaking, the goal is to show, that for any \( \zz \geq 0  \)
\[
\Bigl|\P\Bigl(\sqrt{\| \xi\|^2 + \| \eta\|^2} \leq \zz \Bigr)  - \P^{\b}\Bigl(\sqrt{\| \xi^{{\sbt}}\|^2 + \| \eta^{{\sbt}}\|^2} \leq \zz  \Bigr) \Bigr| \leq \CONST/ \sqrt{h}.
\]
To do that, we first consider separately each
\begin{equation}
\label{def:T_xi}
\Bigl|\P\bigl(\|\xi\|^2  \leq \zz \bigr)  - \P^{\b}\bigl(\|\xi^{{\sbt}}\|^2  \leq \zz  \bigr) \Bigr| 
\end{equation}
and 
\begin{equation}
\label{def:T_eta}
\Bigl|\P\bigl(\| \eta \|^2  \leq \zz \bigr)  - \P^{\b}\bigl(\| \eta^{{\sbt}} \|^2  \leq \zz  \bigr) \Bigr|. 
\end{equation}
Since, both cases~\eqref{def:T_xi} and~\eqref{def:T_eta} are exactly analogous, we investigate only~\eqref{def:T_xi}. 
The obtained result then easily applies to~\eqref{def:T_eta}. 
\begin{table}[!h]
	\caption{Idea of the proof of Lemma~\ref{thm:indBSValidity}}
	\centering	
	\begin{tabular}{l c r  }
		$\|\xi \|  \quad \quad $ & $\underset{C/\sqrt{n}}{\overset{\text{Prop. \ref{theorem:T_approx_real_world}}}{\approx}}$ & $\quad  \|\tilde{\xi}\|$ \\
		&  & $\quad \quad \quad \approxtest$ \\
		$\|\xi^{{\sbt}} \| \quad \quad $ & $\underset{C/\sqrt{n}}{\overset{\text{Prop. \ref{thm:bApprox}}}{\approx}}$ & $\quad  \|\tilde{\xi}^{{\sbt}}\| $ \\
	\end{tabular}
	\label{tabel:boot_proof}
\end{table}
The proof of Lemma~\ref{lemma:klom_dist_conf_sets_barycenters} mainly relies on three steps that are depicted in Table~\ref{tabel:boot_proof},
where \( \tilde{\xi} \) and \(\tilde{\xi}^{{\sbt}} \) are zero-mean Gaussian vectors: 
\begin{equation}
\label{def:gaussians_barycenters}
\tilde{\xi} \sim \mathcal{N}(0,  \mathtt{v}^2_{\xi}), \quad \tilde{\xi}^{{\sbt}} \sim \mathcal{N}(0,  \mathtt{v}^{{\sbt} 2}_{\xi}),
\end{equation}
where \( \mathtt{v}^{{\sbt} 2}_{\xi} \eqdef \frac{1}{n}\sum_{i = 1}^n \xi_i\xi^{\T}_i \).

The three steps yield the following lemma:
\begin{lemma}
	\label{thm:indBSValidity}
	Let $\rho \eqdef \E(\|\xi_i\|^3) < \infty$. Then it holds with probability  \(\P^{{\sbt}} \geq 1- (e^{-\xx} + 2e^{-6\xx^2})\), \(\P \geq 1 - 6e^{-\xx} \)
	\begin{equation}
	\label{eq:boundIndBSVal}
	\sup_{\zz}\bigl|\P(\|\xi\| < \zz) - \P^{{\sbt}}(\|\xi^{{\sbt}}\| < \zz) \bigr| \leq \Delta_{\xi, \text{total}}(n), 
	\end{equation}
	where
	\begin{equation*}
	\Delta_{\xi, \text{total}}(n) \eqdef \Delta_{\xi, \text{GAR}}(n) + 2\Delta^{{\sbt}}_{\xi, \text{GAR}}(n) + \Delta_{\xi, \text{AC}}(n) + \Delta_{\xi, \mathtt{v}^2}(n) \leq \CONST / \sqrt{n},
	\end{equation*}
	with \(\Delta_{\xi, \text{GAR}}(n)\) comes from  \eqref{def:error_gar_barycenters}, 
	\(\Delta^{{\sbt}}_{\xi, \text{GAR}}(n) \) \eqref{def:Delta_gar_boot_barycenters}, 
	\(\Delta_{\xi, \text{AC}}(n)\) \eqref{def:ac_error_barycenters} 
	and  \(\Delta_{\xi, \mathtt{v}^2}(n)\) \eqref{def:error_pinsker_barycenters}.
\end{lemma}

\subsubsection{Proof of Lemma \ref{thm:indBSValidity}}

\paragraph{Step 1: Gaussian approximation of \(\| \xi \| \).}
The first proposition is a result about the closeness of the distributions of the Euclidean norms $\|\xi\|$ and $\|\tilde{\xi}\|$. 
\begin{proposition}[Gaussian approximation]
	\label{theorem:T_approx_real_world}
	Let $\rho_{\xi} \eqdef \E(\|\xi_i\|^3) < \infty$, $i= 1 \dots, n$. Then, 
	\begin{equation}
	\label{eq:GAR_real_world_barycenters}
	\sup_{\zz} |\P( \|\xi\| <\zz) - \P(\|\tilde{\xi}\| <\zz)| \leq  \Delta_{\xi, \text{GAR}}(n),
	\end{equation}
	\begin{equation}
	\label{def:error_gar_barycenters}
	 \Delta_{\xi, \text{GAR}}(n) \eqdef 400 d^{1/4} \rho_{\xi} /(\lambda_d^3\sqrt{n}).
	\end{equation}
	where $\lambda_d$ denotes the square root of the smallest eigenvalue of $\mathtt{v}^2_{\xi}$. 
\end{proposition}
\begin{proof}
	Observe that $\P(\|\xi\| <\zz)= \P(\xi \in C)$ where $C \eqdef \{x \in \R^d \,|\, \|x\| \leq \zz\}$. Since $C$ is convex, the proof follows 
	immediately from Corollary \ref{cor:BentkusCor} taking $S_n^\prime:= \xi$ and $Y^\prime= \tilde{\xi}$.
\end{proof}
\paragraph{Step 2: Gaussian approximation of $\|\xi^{\sbt}\|$.}
The next proposition provides conditions for the approximation of the distribution 
of the Euclidean norm $\|\xi^{{\sbt}}\|$ by the norm $\|\tilde{\xi}^{{\sbt}}\|$.
\begin{lemma}[Approximation of $\|\xi^{\sbt}\|$ by $\|\tilde{\xi}^{\sbt}\|$]
	\label{thm:bApprox}
	Let $\xi^{{\sbt}}$ and $\tilde{\xi}^{\sbt}$ be the random vectors defined in (\ref{def_xi_boot}) and (\ref{def:gaussians_barycenters}).
	 Then it holds, with  \(\P^{{\sbt}} \geq 1- (e^{-\xx} + 2e^{-6\xx^2})\), \(\P \geq 1 - 4e^{-\xx} \)
	\begin{equation}
	\label{eq:bApprox}
	\sup_{\zz}\bigl |\P^{\sbt}(\|\xi^{\sbt}\| < \zz ) - \P^{\sbt}(\|\tilde{\xi}^{\sbt}\| < \zz) \bigr | \leq \Delta_{\xi, \text{AC}}(n),
	\end{equation}
	\begin{equation}
	\label{def:ac_error_barycenters}
	\Delta_{\xi, \text{AC}}(n) \eqdef C \frac{\xx d^{1/4} }{ \sqrt{n}}\sqrt{\frac{\bigl({\lambda}^2_1 + \CONST(\xx + \log d)\bigr)(d + 6\xx)}{\bigl|{\lambda}^2_1 - \CONST(\xx + \log d)\bigr|}}.
	\end{equation}
\end{lemma}
\begin{proof}
	The sketch of the proof is following:
	\begin{enumerate}
		\item[\textbf{Step 2.1}] Show that \(\|\xi^{{\sbt}}\| \approx \| \overline{\xi}^{{\sbt}} \| \), where
			\begin{equation}
			\label{def:xi_interm}
			\overline{\xi}^{{\sbt}} \eqdef \frac{1}{\sqrt{n}}\sum_{i=1}^n (w_i - 1) \xi_i,
			\end{equation}
		\item[\textbf{Step 2.2}] Show that \( \|\overline{\xi}^{{\sbt}}\| \approx \| \tilde{\xi}^{{\sbt}} \| \),
		\item[\textbf{Step 2.3}] Show that from steps 2.1 and 2.2 it follows that \( \|{\xi}^{{\sbt}}\| \approx \| \tilde{\xi}^{{\sbt}} \| \).
	\end{enumerate}
	\paragraph{Step 2.1: \(\|\xi^{{\sbt}}\| \approx \| \overline{\xi}^{{\sbt}} \| \)}
	Due to subexponentiallity of weights \(w_i \) (see {Condition} \( {(E\!D^{W}_{\b})} \)) it holds with probability
	\( \P^{{\sbt}} \geq 1 - 2e^{-6\xx^2} \) that
	\[
	\Biggl| \frac{1}{n}\sum_{i = 1}^n w_i - 1 \Biggr | \leq \xx/\sqrt{n}.
	\]
	Denote as
	\[
	c_{+}(n) \eqdef \frac{1}{1 + \xx/\sqrt{n}}, \quad 	c_{-}(n) \eqdef \frac{1}{1 - \xx/\sqrt{n}},
	\]
	thus we obtain
	\[
	\frac{c_{+}(n)}{n} \leq \frac{1}{\sum_{i}w_i} \leq \frac{c_{-}(n)}{n}.
	\]
	This entails the following system of inequalities, for any \(\zz > 0 \)
	\begin{equation}
	\label{eq:boot_approx_barycenters}
	\P^{{\sbt}}\bigl( c_{-}(n) \| \overline{\xi}^{{\sbt}} \| \leq \zz \bigr) \leq \P^{{\sbt}}\bigl( \| {\xi}^{{\sbt}}\| 
	\leq \zz \bigr) \leq \P^{{\sbt}}\bigl( c_{+}(n) \| \overline{\xi}^{{\sbt}} \| \leq \zz \bigr).
	\end{equation}
	\paragraph{Step 2.2: \(\|\overline{\xi}^{{\sbt}}\| \approx \| \tilde{\xi}^{{\sbt}} \|\)}
	This result is essentially the same as Proposition~\ref{theorem:T_approx_real_world}. 
	Let $\rho^{{\sbt}}_{\xi} \eqdef \E(\|\overline{\xi}^{{\sbt}}_i\|^3) < \infty$, 
	where \(\overline{\xi}^{{\sbt}}_i \eqdef (w_i-1)\xi_i \), then we obtain
	\begin{equation}
	\label{eq:GAR_boot_barycenters}
	\sup_{\zz} \bigl|\P^{{\sbt}}( \|\overline{\xi}^{{\sbt}}\| <\zz) - \P(\|\tilde{\xi}^{{\sbt}}\| <\zz) \bigr| \leq  \Delta^{{\sbt}}_{\text{GAR}},
	\end{equation}
	where 	\(\lambda^{{\sbt}}_d \eqdef \min \lambda(\mathtt{v}^{{\sbt}}_{\xi}) \) and
	\begin{equation}
	\label{def:Delta_gar_boot_barycenters}
	\Delta^{{\sbt}}_{\xi, \text{GAR}}(n)	\eqdef 400 d^{1/4} \rho^{{\sbt}}_{\xi} /(\lambda^{{\sbt} 3}_d\sqrt{n}).
	\end{equation}
	
	\paragraph{Step 2.3: \(\|{\xi}^{{\sbt}}\| \approx \| \tilde{\xi}^{{\sbt}} \|\)}
	Combining \eqref{eq:boot_approx_barycenters} and \eqref{eq:GAR_boot_barycenters}, one obtains 
	\[
	\P^{{\sbt}}\bigl(  \| \tilde{\xi}^{{\sbt}} \|  \leq \zz / c_{-}(n) \bigr) - \Delta^{{\sbt}}_{\xi, \text{GAR}}(n) \leq \P^{{\sbt}}\bigl( \| {\xi}^{{\sbt}}\| 
	\leq \zz \bigr) \leq \P^{{\sbt}}\bigl( c_{+}(n) \| \tilde{\xi}^{{\sbt}} \| 
	\]
	\[
	\leq \zz / c_{+}(n)\bigr) + \Delta^{{\sbt}}_{\xi, \text{GAR}}(n).
	\]
	And, obviously
	\[
	\P^{{\sbt}}\bigl(  \| \tilde{\xi}^{{\sbt}} \|  \leq \zz / c_{-}(n) \bigr) - \Delta^{{\sbt}}_{\xi, \text{GAR}}(n) \leq \P^{{\sbt}}\bigl( \| \tilde{\xi}^{{\sbt}}\| 
	\leq \zz \bigr) 
	\]
	\[
	\leq \P^{{\sbt}}\bigl( \| \tilde{\xi}^{{\sbt}} \| \leq \zz / c_{+}(n)\bigr) + \Delta^{{\sbt}}_{\xi, \text{GAR}}(n).
	\]
	Now note that
	\[
	 \zz / c_{-}(n) = \zz - \frac{\xx}{\sqrt{n}}\zz, \quad   \zz / c_{+}(n) = \zz + \frac{\xx}{\sqrt{n}}\zz.
	\]
	Thus
	\begin{align}
	\label{eq:GAR_bootstrap_barycenters}
	\Bigl|  \P^{{\sbt}}\bigl( \| {\xi}^{{\sbt}}\| \leq \zz \bigr) &-  \P^{{\sbt}}\bigl( \| \tilde{\xi}^{{\sbt}}\|  \leq \zz \bigr) \Bigr| \leq 	2\Delta^{{\sbt}}_{\xi, \text{GAR}}(n) \nonumber\\
 &+  \P^{{\sbt}}\bigl( \| \tilde{\xi}^{{\sbt}} \| \leq  \zz + \zz\xx/\sqrt{n} \bigr)
 - \P^{{\sbt}}\bigl(  \| \tilde{\xi}^{{\sbt}} \|  \leq \zz - \zz\xx/\sqrt{n} \bigr) . 
	\end{align}
	
	Since $\tilde{\xi}^{{\sbt}}$ is a Gaussian random vector, we can apply the anti-concentration result
	from Lemma \ref{lem:antiConc} in order to bound the right-hand-side of \eqref{eq:GAR_bootstrap_barycenters}. 
	Set $h_1= h_2 = z\tfrac{\mathtt{x}}{\sqrt{n}}$ and 
	\begin{align*}
	& Q^{h_1} = \{x \, |\, \|x\|< z+z\tfrac{\mathtt{x}}{\sqrt{n}}\} \\
	& Q^{-h_2} = \{x \,|\, \|x\|< z-\tfrac{\mathtt{x}}{\sqrt{n}}\}.
	\end{align*}
	Then, 
	\begin{align}
	\begin{split}
	\label{eq:antiConcAppl}
	\P^{{\sbt}}(\|\tilde{\xi}^{{\sbt}}\| &< \zz + \zz\xx/\sqrt{n} ) - \P^{\sbt}(\|\tilde{\xi}^{\sbt}\| < \zz - \zz\xx/\sqrt{n} )    \\
	& = \P^{{\sbt}}(\tilde{\xi}^b \in Q^{h_1} ) - \P^{{\sbt}}(\tilde{\xi}^b \in Q^{-h_2}) \leq C\zz\tfrac{\xx}{\sqrt{n}} \sqrt{\|{\mathtt{v}}^{{\sbt} -2}_{\xi}\|_F}
	\end{split}
	\end{align}
	Note, that \(\mathtt{v}^{{\sbt} -2}_{\xi} \) stands for \(\mathtt{v}^{{\sbt}}_{\xi} \) to the power \(-2 \). 
	The right hand side of inequality (\ref{eq:antiConcAppl}) depends on \(\zz\). 
	In order to obtain a uniform bound we can apply Theorem \ref{thm:ldevGauss}, 
	which states that, for \(\mathtt{p} \eqdef\tr({\mathtt{v}}^{{\sbt} 2}_{\xi})\), 
	\(\mathtt{v} \eqdef 2 \tr({\mathtt{v}}^{{\sbt} 4}_{\xi})\) and the largest eigenvalue of \(\mathtt{v}^{{\sbt} 2}_{\xi} \), 
	\(\hat{\lambda}_1^{2} \eqdef \max\bigl( \lambda({\mathtt{v}}^{{\sbt} 2}_{\xi}) \bigr)\), 
	\begin{equation*}
	\P^{{\sbt}}\bigl(\|\tilde{\xi}^{{\sbt}}\|^2 < \mathtt{p} + (2 \mathtt{v} \mathtt{x}^{1/2}) \vee (6{\lambda}_1^2\mathtt{x}) \bigr) \leq 1-e^{-\mathtt{x}},
	\end{equation*} 
	Define $\hat{\Delta}^2 \eqdef \mathtt{p} + (2 \mathtt{v} \mathtt{x}^{1/2}) \vee (6\hat{\lambda}_1^2\xx)$, 
	we can bound \(	\zz < \hat{\Delta}\). Also note that 
	\begin{equation*}
	\mathtt{p} \leq d\hat{\lambda}_1^2 \qquad \text{and} \qquad \mathtt{v} \leq \sqrt{2d} \hat{\lambda}_1^2 .
	\end{equation*}
	Hence,
	\[
	\hat{\Delta} \leq
	\begin{cases}
	 C \sqrt{\hat{\lambda}^2_1(d + 2\sqrt{2}\xx^{1/2}) } & \mbox{if } \xx < 2d/9 \\ 
	 C \sqrt{\hat{\lambda}^2_1(d + 6\mathtt{x})} & \mbox{if } \xx \geq 2d/9. 
	\end{cases}
	\]
	Hence, we obtain
	\begin{equation*}
	\bigl|\P^{{\sbt}}\bigl(\|\xi^{{\sbt}}\| < \zz \bigr) - \P^{{\sbt}}\bigl(\|\tilde{\xi}^{{\sbt}}\| < \zz\bigr) \bigr| \leq C \frac{\xx}{\sqrt{n}} \hat{\Delta} \sqrt{\bigl\|\mathtt{v}^{{\sbt} -2}_{\xi} \bigr\|_F}.
	\end{equation*}
	
	Now we have to bound two random quantities: \(\hat{\lambda}_1^{2}\) and \( \|\mathtt{v}^{{\sbt} -2}_{\xi} \bigr\|_F \). 
	First note that \(\|\mathtt{v}^{{\sbt} -2}_{\xi} \bigr\|^2_F \leq d (\hat{\lambda}_d^2)^{-2}\), 
	where \(\hat{\lambda}_d^2\) is the smallest eigenvalue of \(\mathtt{v}^{{\sbt} 2}_{\xi}\). 
	We use eigenvalue stability inequality (see e.g.~\citet{tao2012topics}) 
	together with Corollary~\ref{corollary:matrix_bernstein}:
		\[
		\bigl|\hat{\lambda}_d^{2} - {\lambda}_d^{2} \bigr| \leq  \CONST(\xx + \log d).
		\]
	Thus one obtains with probability {\(\P \geq 1 - 2e^{-\xx}\) }
	\[
	 \|\mathtt{v}^{{\sbt} -2}_{\xi} \bigr\|^2_F \leq d (\hat{\lambda}_d^2)^{-2} \leq d \bigl({\lambda}_d^2 -   \CONST(\xx + \log d)\bigr)^{-2}.
	\]
	Furthermore, Corollary~\ref{corollary:matrix_bernstein} ensures, that with \(\P  \geq 1 - 2e^{-\xx} \)
	\[
	\bigl|\hat{\lambda}_1^{2} - {\lambda}_1^{2} \bigr| \leq  \CONST(\xx + \log d).
	\]
	Thus, collecting all bounds we obtain the result
\begin{equation}
\label{eq:ac_approx_barycenters}
	\Bigl|  \P^{{\sbt}}\bigl( \| {\xi}^{{\sbt}}\| \leq \zz \bigr) -  \P^{{\sbt}}\bigl( \| \tilde{\xi}^{{\sbt}}\|  \leq \zz \bigr) \Bigr| \leq 	2\Delta^{{\sbt}}_{\xi, \text{GAR}}(n) + \Delta_{\xi, \text{AC}}(n),
\end{equation}
with \(\Delta_{\xi, \text{AC}}(n) \) comes from~\eqref{def:ac_error_barycenters}
\end{proof}
\paragraph{Step 3: Gaussian comparison}

The remaining problem is the comparison of the distributions of the norms of the two Gaussian random vectors.
\begin{equation}
\label{eq:tildeVecs}
\tilde{\xi} \sim \mathcal{N}(0, \mathtt{v}^2_{\xi}) \qquad \text{and} \qquad\tilde{\xi}^{\sbt} \sim \mathcal{N}(0, \mathtt{v}^{{\sbt} 2}_{\xi}).
\end{equation} 
\begin{proposition}[Gaussian comparison]
	\label{thm:gCompAppl}
	Let $\tilde{\xi}$ and $\tilde{\xi}^{\sbt}$ be the random vectors defined in (\ref{eq:tildeVecs}). Then, it holds with $\P \geq 1-e^{-\mathtt{x}}$
	\begin{equation}
	\label{eq:gcCompAppl}
	\sup_{\zz}\bigl| \P(\|\tilde{\xi}\| <\zz) - \P^{\sbt}(\|\tilde{\xi}^{\sbt}  \|< \zz) \bigr| \leq \Delta_{\mathtt{v}^2}(n),
	\end{equation}
	where
	\begin{equation} 
	\label{def:error_pinsker_barycenters}
	\Delta_{\xi, \mathtt{v}^2}(n) \eqdef C{\frac{d(\xx+\log d)}{\sqrt{n}}}.
	\end{equation}
\end{proposition}
\begin{proof}
	Let \(N_{\xi}= \mathcal{N}(0, \mathtt{v}^2_{\xi})\) and \(N^{{\sbt}}_{\xi}= \mathcal{N}(0, \mathtt{v}^{{\sbt} 2}_{\xi})\) denote the laws of $\tilde{\xi}$ and $\tilde{\xi}^{\sbt}$. Then, we have the inequality
	\begin{equation}
	\label{eq:TVdist}
	\sup_\zz \bigl|\P(\|\tilde{\xi}\| < \zz) - \P(\|\tilde{\xi}^{\sbt}\| < \zz) \bigr| \leq
	\sup_{B \in \mathcal{B}(\R^d)}\bigl| N_{\xi}(B)-N_{\xi}^{\sbt}(B)\bigr|.
	\end{equation}
	The right hand side of (\ref{eq:TVdist}) is the Total Variation distance between the measures $N_{\xi}$ and $N_{\xi}^{\sbt}$. By Pinsker's inequality, (Lemma \ref{lem:pinsker}), it can be bounded by the Kullback-Leibler divergence, in the following way:
	\begin{equation}
	\label{eq:Pinsker}
	\sup_{B \in \mathcal{B}(\R^d)} \bigl| N_{\xi}(B)-N_{\xi}^{\sbt}(B)\bigr| \leq \sqrt{\mathcal{KL}(N_{\xi}, N_{\xi}^{\sbt})/2}.
	\end{equation}
	Therefore, we can proof the Theorem by bounding $\mathcal{KL}(N_{\xi}, N_{\xi}^{\sbt})$. Denote as
	\[
	Z_{\xi} \eqdef \mathtt{v}^{-1}_{\xi}\mathtt{v}^{{\sbt} 2}_{\xi}\mathtt{v}^{-1}_{\xi} - I_d.
	\]
	Lemma \ref{lem:gaussComp} provides necessary conditions for a bound. In particular, we need to find a $\rho_{V^2}^2 \geq 0$, such that
	\begin{equation}
	\label{eq:covCondappl}
	\bigl\| Z_{\xi}\bigr\| \leq 1/2 \qquad \text{and} \qquad
	\tr\bigl(Z^2_{\xi} \bigr) \leq \rho_{V}^2.
	\end{equation}
	A natural bound on the trace is
	\[
	\tr\bigl(Z^2_{\xi} \bigr) = \| Z_{\xi}\|^2_{F} \leq d(\alpha^2_1)^2,
	\]
	where \(\alpha^2_1 \) is the largest eigenvalue of \(Z_{\xi} \). Corollary~\ref{corollary:matrix_bernstein} implies, that with probability \(\P \geq 1 - 2e^{-\xx} \)
	\[
	(\alpha^2_1)^2 \leq \CONST(\xx + \log d)^2/{n}.
	\]
	Applying the result of Lemma~\ref{lem:gaussComp}, one obtains
		\begin{equation*}
		\mathcal{KL}(N_{\xi}, N_{\xi}^{\sbt}) \leq d  \CONST(\xx + \log d)^2/{n}
		\end{equation*} 
		and by Pinsker's inequality (\ref{eq:Pinsker}) we obtain (\ref{eq:gcCompAppl}).
\end{proof}

Collecting the bounds \eqref{def:error_gar_barycenters}, \eqref{def:Delta_gar_boot_barycenters}, \eqref{def:ac_error_barycenters} and \eqref{def:error_pinsker_barycenters} finally yields the bound (\ref{eq:boundIndBSVal}) and thus proofs Theorem \ref{thm:indBSValidity}.

\subsubsection{Proof of Lemma \ref{lemma:klom_dist_conf_sets_barycenters}}

We are now able to proof Lemma \ref{lemma:klom_dist_conf_sets_barycenters}. 
\begin{proof}[Proof of Lemma \ref{lemma:klom_dist_conf_sets_barycenters}]
	We use Lemma \ref{thm:indBSValidity} by which we obtain for the bootstrap approximation of the statistics $\|\xi \|^2$ and $\| \eta \|^2$, 
	\begin{align}
	\begin{split}
	\label{eq:indBSapprox}
	\sup_\zz & \bigl|\P( \|\xi\|^2 < \zz^2 ) - \P^{\sbt}( \|\xi^{{\sbt}}\|^2 < \zz^2 ) \bigr| \leq \Delta_{\xi, \text{total}}(n)\\
	\sup_\zz & \bigl|\P( \| \eta \|^2 < \zz^2 ) - \P^{\sbt}(\| \eta^{{\sbt}} \|^2  <\zz^2 )\bigr| \leq \Delta_{\eta, \text{total}}(n)
	\end{split}
	\end{align}
	We need to provide a bound for the distance 
	\begin{equation*}
	\sup_\zz \Bigl|\P(\sqrt{\|\xi \|^2+\|\eta \|^2} <\zz) - \P^{{\sbt}} (\sqrt{\|\xi^{{\sbt}} \|^{2}+ \|\eta^{{\sbt}} \|^{2}} <\zz) \Bigr|.
	\end{equation*}
	Note, that the distribution of the sum of two independent variables is the convolution of the individual distribution functions. Then, by (\ref{eq:indBSapprox}), we obtain for all $\zz \geq 0$, 
	\begin{align}
	\P(\|\xi \|^2  +  \|\eta \|^2 < \zz^2) 
	& = \int \P(\|\eta \|^2 < \zz^2-x) f_{\|\xi \|^2}(x) dx \nonumber\\
	& \leq \int \P^{\sbt}(\|\eta^{{\sbt}} \|^{2} < \zz^2-x)f_{\|\xi \|^2}(x) dx +  \Delta_{\eta, \text{total}}(n) \\
	& = \int \P(\|\xi \|^2 < \zz^2 - y^{\sbt})f_{\|\eta^{{\sbt}} \|^{2}}(y) dy+ \Delta_{\eta, \text{total}}(n) \nonumber\\
	\label{step6}
	& \leq \int \P^{\sbt}(\|\xi \|^{{\sbt} 2} < \zz^2- y^{\sbt})f_{\|\eta^{{\sbt}} \|^{2}}(y) dy  \\
	& + \Delta_{\xi, \text{total}}(n) + \Delta_{\eta, \text{total}}(n) \nonumber \\
	&= \P^{\sbt}(\|\xi^{{\sbt}} \|^{2} +\|\eta^{{\sbt}} \|^{2} < \zz^2) + \Delta_{\xi, \text{total}}(n) + \Delta_{\eta, \text{total}}(n) \nonumber.
	\end{align}
	And in the inverse direction,
	\begin{align}
	\P(\|\xi \|^2  +  \|\eta \|^2 < \zz^2) 
	\geq  \P^{\sbt}(\|\xi^{{\sbt}} \|^{ 2} +\|\eta^{{\sbt}} \|^{2} < \zz^2) - \bigl(\Delta_{\xi, \text{total}}(n) + \Delta_{\eta, \text{total}}(n)\bigr) \nonumber.
	\end{align}
	Which yields, for all $\zz >0$, and $\Delta_{\text{total}}(n) \eqdef \bigl(\Delta_{\xi, \text{total}}(n) + \Delta_{\eta, \text{total}}(n)\bigr)$
	\begin{equation*}
	\Bigl|\P(\sqrt{\|\xi \|^2+\|\eta \|^2} < \zz) - \P^{\sbt}(\sqrt{\|\xi^{{\sbt}} \|^{2}+\|\eta^{{\sbt}} \|^{2}} < \zz) \Bigr| \leq \Delta_{\text{total}}(n).
	\end{equation*}
	This proves the theorem. 
\end{proof}

\subsubsection{Proof of Theorem \ref{theorem:bootstrap_validity_final_barycenters}}
Now we are ready to present the proof of the main result.
\begin{proof}[Proof of Theorem~\ref{theorem:bootstrap_validity_final_barycenters}]
Lemma \ref{lemma:klom_dist_conf_sets_barycenters} implies that for any \(\zz > 0 \)
\[
\bigl| \P\bigl(T_n \leq \zz \bigr) - \P^{{\sbt}}\bigl(T^{{\sbt}}_n \leq \zz\bigr) \bigr| \leq \Delta_{\text{total}}(n).
\]
In other words
\[
\zz_{n}\bigl(\alpha + \Delta_\text{total}(n)\bigr) \leq \zz^{{\sbt}}_{ n}(\alpha) \leq \zz_{n}\bigl(\alpha - \Delta_\text{total}(n)\bigr).
\]
Taking into account, that \(\P \) is a continuous function of quantile \(\zz \), one obtains
\[
\P\bigl(T_n > \zz^{{\sbt}}_n(\alpha)\bigr) - \alpha \geq \P\bigl( T_n > \zz_{n}(\alpha + \Delta_\text{total}(n)) \bigr) - \alpha \geq -\Delta_{\text{total}}(n).
\]
By analogy,
\[
\P\bigl(T_n > \zz^{{\sbt}}_n(\alpha)\bigr) - \alpha \leq \Delta_{\text{total}}(n).
\]
\end{proof}

\section{Validity of the bootstrap procedure for change point detection}
\label{section:boot_validity_cp}
In this section we mainly borrow ideas of the proof from Section~\ref{subsection:boot_validity}. 
We further assume that a training sample in hand is homogeneous
and does not contain change points.
For transparency of 
First let time moment \(t \) be fixed. The test statistic~\eqref{def:cp_test_commute} 
and its bootstrapped counterpart~\eqref{def:cp_test_commute_boot} are written as
\[
T^2(t) = \|{\xi}(t)\|^2 + \|{\eta}(t)\|^2,
\]
where
\[
\xi(t) \eqdef \frac{1}{\sqrt{h}} \sum_{i \in \{\Lt, \Rt \} } c_i(t)\xi_i, \quad \eta(t) \eqdef \frac{1}{\sqrt{h}} \sum_{i \in \{\Lt, \Rt \} } c_i(t)\eta_i.
\]
\begin{equation}
\label{def:coeff_c_barycenters}
c_i(t) = 
\begin{cases}
1, &~\text{if}~ i \in \Lt\\
-1, &~\text{if}~ i \in \Rt
\end{cases}
\end{equation}
and
\[
T^{{\sbt} 2}(t) =  \|{\xi}^{{\sbt}}(t)\|^2 + \|{\eta}^{{\sbt}}(t)\|^2,
\]
\begin{equation}
\label{def:xi_boot_cp_barycenters}
{\xi}^{{\sbt}}(t) \eqdef \Biggl\| \frac{\sqrt{h}}{\sum_{i \in \Lt}w_i} \sum_{i \in \Lt } w_i\xi_i - \frac{\sqrt{h}}{\sum_{i \in \Rt}w_i} \sum_{i \in \Rt } w_i\xi_i  \Biggr \|^2,
\end{equation}
\begin{equation}
{\eta}^{{\sbt}}(t) \eqdef \Biggl\| \frac{\sqrt{h}}{\sum_{i \in \Lt}w_i} \sum_{i \in \Lt } w_i\eta_i - \frac{\sqrt{h}}{\sum_{i \in \Rt}w_i} \sum_{i \in \Rt } w_i\eta_i \Biggr \|^2. 
\end{equation}
Here \(\xi_i \) and \(\eta_i \) come from~\eqref{def:xi} and~\eqref{def:eta} respectively.
We consider each term \( \|\xi(t)\| \)  and \( \|\eta(t)\| \) separately and  following Scheme~\ref{tabel:boot_proof} show that
\begin{lemma}
	\label{theorem:approx_cp_barycenters_general}
	Fix time moment \(t\) and let $\rho \eqdef \E(\|\xi_i\|^3) < \infty$. 
	Then it holds with \( \P \geq 1 - 6e^{-\xx} - 2e^{-6\xx^2} \), 
	\(\P^{{\sbt}} \geq 1 -3e^{-\xx} - 2e^{-6\xx^2} \), 
	\begin{equation}
	\label{eq:approx_cp_barycenters}
	\sup_{\zz} \bigl|\P(\|\xi(t)\| < \zz) - \P^{{\sbt}}(\|\xi^{{\sbt}} (t)\| < \zz) \bigr | \leq \Delta_{\xi(t), \text{total}}, 
	\end{equation}
	where
	\begin{align}
	\begin{split}
		\label{def:delta_xi_total_cp_barycenters}
		\Delta_{\xi(t), \text{total}} &\eqdef \Delta_{\xi(t), \text{GAR}}  + 2\Delta^{{\sbt}}_{\xi(t),\text{GAR}}(h) 
		+ \Delta_{\xi(t),\text{AC}}(h) \\
		 &+ \Delta^{{\sbt}}_{\xi(t),\E}(h) + \Delta_{\xi(t), \mathtt{v}^2}(h) \leq \CONST/\sqrt{h},
	\end{split}
	\end{align}
		with \(\Delta_{\xi(t), \text{GAR}} \) from~\eqref{def:gar_error_cp_barycenters},
		\( \Delta^{{\sbt}}_{\xi(t),\text{GAR}}(h) \) ~\eqref{def:Delta_gar_boot_barycenters},
		\(\Delta_{\text{AC}}(h) \)~\eqref{def:ac_error_barycenters},
		\(\Delta^{{\sbt}}_{\E}(h)\)~\eqref{def:bias_error_cp_barycenters},
		\(\Delta_{\xi(t),\mathtt{v}^2}(h)\)~\eqref{def:error_pinsker_barycenters}.
\end{lemma}

\subsection{Proof of Proposition~\ref{theorem:approx_cp_barycenters_general}}

\paragraph{Step 1: Gaussian approximation of \(\| \xi(t) \| \).}
\begin{corollary}
	\label{teorem:approximation_cp_barycenters}
	Let $\rho_{\xi} \eqdef \E(\|\xi_i\|^3) < \infty$. Then it holds 
	\begin{equation*}
		\sup_{\zz} \bigl|\P( \|\xi(t)\| <\zz) - \P(\|\tilde{\xi}(t)\| <\zz)\bigr| \leq  400 d^{1/4} \rho_{\xi} /(\lambda_d^3\sqrt{h}),
	\end{equation*}
		where $\lambda_d$ denotes the square root of the smallest eigenvalue of $\mathtt{v}^2_{\xi(t)}$ and
		\begin{equation}
		\label{def:gar_error_cp_barycenters}
		\Delta_{\xi(t), \text{GAR}} \eqdef 400 d^{1/4} \rho_{\xi} /(\lambda_d^3\sqrt{h}).
		\end{equation}
\end{corollary}
\begin{proof}
Note, that one can consider \( \xi(t)\) as a sum of \(h \) iid symmetric random variables. In particular
\[
\xi(t) = \frac{1}{\sqrt{h}}\sum_{i = t - h}^{t + h - 1}(\xi_i - \xi_{i + h}).
\]
The rest follows immediately from Proposition~\ref{theorem:T_approx_real_world}.
\end{proof}

\paragraph{Step 2: Gaussian approximation of $\|\xi^{\sbt}(t)\|$.}
Now we define three auxiliary construction. Let  \(\overline{\xi}^{{\sbt}}(t) \) be 
\begin{align}
\begin{split}
\label{def:Delta_bias}
\overline{\xi}^{{\sbt}}(t) &\eqdef \frac{1}{\sqrt{h}}\sum_{i \in \{\Lt, \Rt \}}c_i(t)w_i\xi_i\\ 
\E^{{\sbt}} \overline{\xi}^{{\sbt}}(t) &\eqdef \delta^{{\sbt}}_{\E} = \frac{1}{\sqrt{h}}\sum_{i \in \{\Lt, \Rt\}}c_i(t)m_i.
\end{split}
\end{align}
Its non-centred Gaussian counterpart \(\overline{\xi}^{{\sbt}}(t)\) is
\begin{equation}
\label{def:non_centr_gauss_cp_barycenters}
\Var\bigl(\overline{\xi}^{{\sbt}}(t)\bigr) \eqdef \mathtt{v}^{{\sbt} 2}_{\xi(t)} = \frac{1}{h}\sum_{i \in \{\Lt, \Rt \}}\xi_i\xi^{\T}_i,
\quad
\breve{\xi}^{{\sbt}}(t) \sim \mathcal{N}\bigl(\delta^{{\sbt}}_{\E}, \mathtt{v}^{{\sbt} 2}).
\end{equation}
And the last key ingredient is centred Gaussian vector with the same covariance structure as \(\breve{\xi}^{{\sbt}}(t)\)
\begin{equation}
\label{def:gaussian_cp_bootstrap}
\tilde{\xi}^{{\sbt}}(t) \sim \mathcal{N}\bigl(0, \mathtt{v}^{{\sbt} 2}).
\end{equation}
\begin{lemma}[Approximation of $\|\xi^{{\sbt}}(t)\|$ by $\|\tilde{\xi}^{\sbt}(t)\|$]
	\label{proposition:gar_boot_cp}
	Let $\xi^{{\sbt}}(t)$ and $\tilde{\xi}^{{\sbt}}(t)$ be the random vectors defined in~\eqref{def:xi_boot_cp_barycenters} and \eqref{def:gaussian_cp_bootstrap}.
	Then it holds, with \( \P \geq 1 - 4e^{-\xx} - 2e^{-6\xx^2} \), \(\P^{{\sbt}} \geq 1 -( e^{-\xx} + 2e^{-6\xx^2}) \)
	\begin{equation}
	\sup_\zz|\P^{\sbt}(\|\xi^{\sbt}(t)\| < \zz ) - \P^{\sbt}(\|\tilde{\xi}^{\sbt}(t)\| < \zz) | \leq \Delta^{{\sbt}}_{\xi(t)}(h),
	\end{equation}
	where 
	\begin{equation}
	\Delta^{{\sbt}}_{\xi(t)}(h) \eqdef 2\Delta^{{\sbt}}_{\xi(t),\text{GAR}}(h) + \Delta_{\xi(t),\text{AC}}(h) + \Delta^{{\sbt}}_{\xi(t),\E}(h)  \leq \CONST/\sqrt{h},
	\end{equation}
	with \( \Delta^{{\sbt}}_{\xi(t),\text{GAR}}(h) \) from~\eqref{def:Delta_gar_boot_barycenters},
	\(\Delta_{\text{AC}}(h) \)~\eqref{def:ac_error_barycenters},
	\(\Delta^{{\sbt}}_{\E}(h)\)~\eqref{def:bias_error_cp_barycenters}.
\end{lemma}
\begin{remark}
Note, that the difference with the result of Lemma~\ref{thm:bApprox} consists of 
an additional error term \(\Delta^{{\sbt}}_{\E}(h)\), that appears in the bootstrap world 
due to relative bias \(\delta^{{\sbt}}_{\E}(h)\)~\eqref{def:Delta_bias} 
between estimators in left and right halves of a running window.
\end{remark}
\begin{proof}
Now, following Step 2.1 and Step 2.2 in the proof of Proposition~\ref{thm:bApprox} obtain the results, similar to~\eqref{eq:boot_approx_barycenters} and~\eqref{eq:GAR_boot_barycenters}:
\begin{equation*}
\label{eq:boot_cp_approx_barycenters}
\P^{{\sbt}}\bigl( c_{-}(n) \| \overline{\xi}^{{\sbt}}(t) \| \leq \zz \bigr) \leq \P^{{\sbt}}\bigl( \| {\xi}^{{\sbt}}(t)\| 
\leq \zz \bigr) \leq \P^{{\sbt}}\bigl( c_{+}(n) \| \overline{\xi}^{{\sbt}}(t) \| \leq \zz \bigr),
\end{equation*}
\begin{equation*}
	\label{eq:GAR_boot_cp_barycenters}
	\sup_{\zz} \bigl|\P^{{\sbt}}( \|\overline{\xi}^{{\sbt}}(t)\| <\zz) - \P^{{\sbt}}(\|\breve{\xi}^{{\sbt}}(t)\| <\zz) \bigr| \leq  \Delta^{{\sbt}}_{\xi(t),\text{GAR}}(h),
\end{equation*}
with \(\Delta^{{\sbt}}_{\xi(t),\text{GAR}}(h)\) comes from~\eqref{def:Delta_gar_boot_barycenters}.
Combining this two results we come to
	\begin{align*}
	\Bigl|  \P^{{\sbt}}\bigl( \| {\xi}^{{\sbt}}(t)\| \leq \zz \bigr) &-  \P^{{\sbt}}\bigl( \| \breve{\xi}^{{\sbt}}(t)\|  \leq \zz \bigr) \Bigr| \leq 	2\Delta^{{\sbt}}_{\xi(t),\text{GAR}}(h) \nonumber\\
	&+  \P^{{\sbt}}\bigl( \| \breve{\xi}^{{\sbt}}(t) \| \leq  \zz + \zz\xx/\sqrt{h} \bigr)
	- \P^{{\sbt}}\bigl(  \| \breve{\xi}^{{\sbt}}(t) \|  \leq \zz - \zz\xx/\sqrt{h} \bigr) . 
	\end{align*}
Step 2.3 yields the result similar to~\eqref{eq:ac_approx_barycenters}:
\[
\Bigl|  \P^{{\sbt}}\bigl( \| {\xi}^{{\sbt}}(t)\| \leq \zz \bigr) -  \P^{{\sbt}}\bigl( \| \breve{\xi}^{{\sbt}}(t)\|  \leq \zz \bigr) \Bigr| \leq 2\Delta^{{\sbt}}_{\xi(t),\text{GAR}}(h) + \Delta^{{\sbt}}_{\text{AC}}(h)
\]
The last step is to apply anti-concentration result to \(\breve{\xi}^{{\sbt}}(t) \). 
Note that \( \| \breve{\xi}^{{\sbt}}(t) \| = \| \tilde{\xi}^{{\sbt}}(t) + \delta^{{\sbt}}_{\E} \| \). Then using triangle inequality
\[
\P^{{\sbt}}\bigl(\| \tilde{\xi}^{{\sbt}}(t) \| \leq \zz\bigr)  - \P^{{\sbt}}\bigl(  \| \tilde{\xi}^{{\sbt}}(t) + \delta^{{\sbt}}_{\E} \| \leq \zz \bigr) 
\leq 
\P^{{\sbt}}\bigl(\| \tilde{\xi}^{{\sbt}}(t) \| \leq \zz\bigr) -  \P^{{\sbt}}\bigl(\| \tilde{\xi}^{{\sbt}}(t) \| \leq \zz -\| \delta^{{\sbt}}_{\E} \|\bigr).
\]
Therefore, it holds 
\[
\P^{{\sbt}}\bigl(\| \tilde{\xi}^{{\sbt}}(t) \| \leq \zz\bigr) -  \P^{{\sbt}}\bigl(\| \tilde{\xi}^{{\sbt}}(t) \| \leq \zz -\| \delta^{{\sbt}}_{\E} \|\bigr) \leq \CONST \| \delta^{{\sbt}}_{\E} \| \sqrt{\|{\mathtt{v}}^{{\sbt} -2}_{\xi(t)}\|_F}.
\]
Applying the bound  \(\|\mathtt{v}^{{\sbt} -2}_{\xi} \bigr\|_F \leq \sqrt{d}/\epsilon \), we get
\[
\CONST \| \delta^{{\sbt}}_{\E} \| \sqrt{\|{\mathtt{v}}^{{\sbt} -2}_{\xi(t)}\|} \leq \CONST \| \delta^{{\sbt}}_{\E} \|  d^{1/4} \bigl|{\lambda}_d^2 -   \CONST(\xx + \log d)\bigr|^{-1/2}.
\]
However, we are not done since \(\| \delta^{{\sbt}}_{\E} \| \) is a random object under measure \(\P \), see~\eqref{def:Delta_bias}.
As soon as the sum of subexponentials is subexponential too, we can apply Lemma~\ref{lemma:sub_exp} 
together with Condition \((ED^{W}_{b}) \) and obtain with \(\P \geq 1 - {2e^{-6\xx^2}} \), that
\begin{equation}
\label{def:bias_error_cp_barycenters}
 \| \delta^{{\sbt}}_{\E} \| \leq  \Delta^{{\sbt}}_{\xi(t),\E}(h), \quad  \Delta^{{\sbt}}_{\xi(t),\E}(h) \eqdef \frac{\xx}{\sqrt{\nu}\sqrt{h}}.
\end{equation}

Thus, we obtain
\[
\Bigl|  \P^{{\sbt}}\bigl( \| {\xi}^{{\sbt}}(t)\| \leq \zz \bigr) -  \P^{{\sbt}}\bigl( \| \tilde{\xi}^{{\sbt}}(t)\|  \leq \zz \bigr) \Bigr| \leq 2\Delta^{{\sbt}}_{\xi(t),\text{GAR}}(h) + \Delta_{\xi(t),\text{AC}}(h) + \Delta^{{\sbt}}_{\xi(t),\E}(h).
\]
Together with Lemma~\ref{thm:gCompAppl} it yields the final bound
\begin{align}
\begin{split}
\Bigl|  \P\bigl( \| {\xi}(t)\| \leq \zz \bigr) &-  \P^{{\sbt}}\bigl( \| {\xi}^{{\sbt}}(t)\|  \leq \zz \bigr) \Bigr| 
\leq 2\Delta^{{\sbt}}_{\xi(t),\text{GAR}}(h)\\ 
&+ \Delta_{\xi(t),\text{AC}}(h) + \Delta^{{\sbt}}_{\xi(t),\E}(h) + \Delta_{\xi(t), \mathtt{v}^2}(h).
\end{split}
\end{align}
\end{proof}

\subsection{Proof of Theorem~\ref{theorem:bootstrap_validity_final_cp_barycenters}}
Proposition~\ref{theorem:approx_cp_barycenters_general} together with Lemma~\ref{lemma:klom_dist_conf_sets_barycenters}
allow to obtain the bootstrap validity result for change point detection.

\begin{proof}[Proof of Theorem~\ref{theorem:bootstrap_validity_final_cp_barycenters}]
	First note, that training on non-intersected running windows yields the following distribution of \(\max_{t}T(t) \)
	\[
	\P\Bigl( \max_{1 \leq t \leq M}T_h(t) \leq \zz  \Bigr) = \sum_{t \in \mathcal{I}} \P\bigl( T_h(t) \leq \zz  \bigr).
	\]
	For simplicity we assume, that the size of training sample is divisible 
	by the window length \(2h\). Denote as \(M_{h} \eqdef M/2h\). Then the above equality 
	can be continued as
	\[
	\sum_{t \in \mathcal{I}} \P\bigl( T_h(t) \leq \zz  \bigr) = M_{h}\P\bigl( T_h \leq \zz  \bigr).
	\]
	The same relation holds for the bootstrap world
		\[
		\P^{{\sbt}}\Bigl( \max_{1 \leq t \leq M}T^{{\sbt}}_h(t) \leq \zz  \Bigr)  = M_{h}\P^{{\sbt}}\bigl( T^{{\sbt}}_h \leq \zz  \bigr).
		\]
	Applying Lemma~\eqref{lemma:klom_dist_conf_sets_barycenters}, one obtains that
	\begin{align*}
	\sup_{\zz > 0}&\Bigl| \P(\max_{t}T_h(t) \leq \zz) - \P^{{\sbt}}(\max_{t}T^{{\sbt}}_h(t) \leq \zz) \Bigr| \\
	 &\leq M_{h}\bigl| \P\bigl( T_h \geq \zz  \bigr) - \P^{{\sbt}}\bigl( T^{{\sbt}}_h \geq \zz  \bigr)\bigr| \leq M_{h}\Delta_{\text{total, cp}}(h),
	\end{align*}
	where 
	\begin{equation}
	\label{def:error_total_cp_barycenters}
	\Delta_{\text{total, cp}}(h) \eqdef \Delta_{\xi(t), \text{total}}(h) + \Delta_{\eta(t), \text{total}}(h) 
	\end{equation}
	and \(\Delta_{\xi(t), \text{total}}(h)  \), \(\Delta_{\eta(t), \text{total}}(h) \) come from~\eqref{def:delta_xi_total_cp_barycenters}.
	
	Now applying the same line of reasoning as in the proof of Theorem~\ref{theorem:bootstrap_validity_final_cp_barycenters}, we obtain the result.
\end{proof}

\section{Auxiliary results for the proof of the bootstrap validity}
\label{section:boot_barycenters_aux}

\begin{statement}
\label{statement:W2_distance_gauss}
Consider zero-mean measures \(\mu, \nu  \in \mathcal{P}^{ac}_{2}(\R^d)\)
with covariances \(U \) and \(V\) respectively.
We assume that there exists the optimal transportation plan 
from \(\mu \) to \(\nu \) is \(OT(x) = Ax \)
with 
\[
 A(U, V) = U^{-1/2}\bigl( U^{1/2} V U^{1/2} \bigr)^{1/2} U^{-1/2}.
\]
Then 
\[
W^2(\mu, \nu) = \| \bigl(A(U, V) - I \bigr)U^{1/2} \bigr\|^2_{\text{F}}.
\]
\end{statement}
\begin{proof}
By definition of Frobenious norm
\begin{align*}
\bigl \| \bigl(A(U, V) - I \bigr)U^{1/2} \bigr\|^2_{\text{F}} &= \tr\Bigl( U^{1/2} \bigl(A(U, V) - I \bigr)^{\T}\bigl(A(U, V) - I \bigr)U^{1/2} \Bigr)\\
& = \tr \bigl(U + V - 2( U^{1/2} V U^{1/2})^{1/2} \bigr).
\end{align*}
\end{proof}

\begin{statement}
\label{statement:population_barycenter_equiv}
	Let \( \nu_{\omega} \sim \P\) be a random measure with mean \(m_{\omega} \)
	and covariance \(S_{\omega} \).
	We assume that it comes from a class of affine admissible 
	deformations~\eqref{def:admiss_transf} of \(\mu_0 \).
	Let \(A_{\omega} \) be s.t. 
	\[
	A_{\omega} = \Qv^{-1/2}_{0}\bigl(\Qv^{1/2}_{0} S_{\omega} \Qv^{1/2}_{0} \bigr)^{1/2}\Qv^{-1/2}_{0}, \quad \E_{\omega}A_{\omega} = I_{d},
	\]
	and \(a_{\omega} \) s.t. \( E_{\omega} = 0 \).
	Then the population barycenter \(\mu^* \) of \(\P \) coincides with the template object\(\mu_0 \).
\end{statement}
\begin{proof}
	Note, that condition on \(A_{\omega} \) implies
	\[
	\Qv_{0} = \int_{\omega}\bigl(\Qv^{1/2}_{0} S_{\omega} \Qv^{1/2}_{0} \bigr)^{1/2}d\P(\omega).
	\]
	Next, according to the model observed mean \(m_{\omega} \) follows
	\[
	\E_{\omega}m_{\omega} = \E_{\omega}A_{\omega}m_0 + \E_{\omega} a_{\omega} = m_0.
	\]
	Thus, parameters \(m_0 \) and \(\Qv_0 \) coincide with \(m^* \) and \(\Qv^* \), that come from Proposition~\ref{proposition:alvarez}.
	Thus the measures are similar from  at least) \(2\)-Wasserstein distance-point-of-view. 
\end{proof}

The following results are used in the proof of Theorem \ref{thm:indBSValidity}. 

\subsection{Gaussian approximation}

Consider a sample $(X_i)_{1 \leq i \leq n}$ of independent random vectors in $\R^d$, with common mean $\E(X_i)=0$ and $Var(X_i) = I$. Write $\beta=\E(\|X_i\|^3)$. Also. let $Y \sim \mathcal{N}(0, I)$ be a standard normal random vector . Define,
\begin{equation*}
S_n \eqdef \frac{1}{\sqrt{n}} \sum_{i=1}^n X_i
\end{equation*} 
The following theorem provides an upper bound for the difference in distribution between $S_n$ and $Y$ on the class of convex sets.  
\begin{theorem}[Gaussian approximation, Theorem 1.1. in \citet{bentkus_2003}]
	\label{thm:BentkusGAR}
	Let $\mathscr{C}$ be the class of convex subsets of $\R^d$. Then 
	\begin{equation*}
	\sup_{C \in \mathscr{C}} |\P( S_n \in C) - \P(Y \in C)| \leq 400 d^{1/4} \beta/\sqrt{n} 
	\end{equation*}
\end{theorem}
Now, consider the iid random vectors $(X_i^\prime)_{1 \leq i \leq n}$ with $\E(X_i^\prime) = m $ and $\Var(X_i^\prime)= \Sigma$ and write $\beta^\prime= E(\|X_i^\prime\|^3)$. Also let $Y^\prime \sim \mathcal{N}(m, \Sigma)$ and 
\begin{equation*}
S_n^\prime \eqdef \frac{1}{\sqrt{n}} \sum_{i=1}^n X_i^\prime
\end{equation*}
The results from Theorem \ref{thm:BentkusGAR} can be extended to this setting. This is done in the following Corollary.
\begin{corollary}
	\label{cor:BentkusCor}
	Let $\mathscr{C}$ be the class of convex subsets of $\R^d$. Then 
	\begin{equation*}
	\sup_{C \in \mathscr{C}} |\P( S_n^\prime \in C) - \P(Y^\prime\in C)| \leq 400 d^{1/4} \beta^\prime/(\lambda_d^3 \sqrt{n}),
	\end{equation*}
	where $\lambda_d$ denotes the square root of the smallest Eigenvalue of $\Sigma$.
\end{corollary}
\begin{proof}
	Define the symmetric positive matrix $T^2 = \Sigma^{-1}$. Note, that 
	$S_n^\prime=T^{-1}S_n$ and $Y^\prime = T^{-1}Y$.  By changing variables $x \mapsto Tx$ we have. 
	\begin{equation*}
	\P(S_n^\prime \in A)- \P(Y^\prime \in A)= \P(S_n \in TA) - \P(Y \in TA)
	\end{equation*}
	Note, that random vectors $(TX_i)_{1 \leq i \leq n}$ satisfy standard normalization, that is, 
	\begin{equation*}
	\E(TX_i)=Tm \qquad Var(TX_i)= I \qquad i= 1, \dots, n
	\end{equation*}
	Hence, we can use Theorem \ref{thm:BentkusGAR} and the fact the s
	et of convex sets in $\R^d$ is invariant under affine symmetric transformations
	\(TC + t \in \mathscr{C} \), if \(t \in \R^d \), \( T : \R^d \rightarrow \R^d\) is a linear 
	symmetric invertible operator, to obtain
	\begin{equation*}
	\sup_{C \in \mathcal{C}}| \P(S_n^\prime \in C)- \P(Y^\prime \in C)| = \sup_{C \in \mathcal{C}}| \P(S_n \in C)- \P(Y \in C)|  \leq 400 d^{1/4} \beta/\sqrt{n}. 
	\end{equation*}
	In addition, note that $\beta= \E(\|TX^\prime\|^3)$. By definition of the operator norm, 
	\begin{equation*}
	\|TX^\prime\| \leq \|T\|\|X^\prime\|.
	\end{equation*}
	Also, $\|T\|= \lambda_d$ where $\lambda_d^2$ is the smallest Eigenvalue of $\Sigma$. Therefore,
	\begin{equation*}
	\beta \leq \lambda_d^{-3}\beta^\prime, 
	\end{equation*}
	and finally
	\begin{equation*}
	\sup_{C \in \mathscr{C}} |\P( S_n \in C) - \P(Y \in C)| \leq 400 d^{1/4} \beta^\prime/(\lambda_d^3 \sqrt{n})
	\end{equation*}
\end{proof}

\subsection{Anti-concentration inequality for a Gaussian vector}

Next, we will state an anti-concentration result for a Gaussian random vector on a convex set in $\R^d$. The following result is used in Step $2$ of the bootstrap proof (Proposition \ref{thm:bApprox}). 
\begin{lemma}[Lemma A.2 in \citet{chernozhukov_2014}]
	\label{lem:antiConc}
	Let $Y$ be the Gaussian random vector $Y \sim \P \eqdef \mathcal{N}(0, \Sigma)$. Then there exists an absolute constant $C >0$ such that for every convex set $C \subset \R^d$, and every $h_1, h_2 >0$, 
	\begin{equation*}
	\P(C^{h_1} \backslash C^{-h_2}) = \P(Y \in C^{h_1} ) - \P(Y \in C^{-h_2})  \leq C \sqrt{\|\Sigma^{-1}\|_F}(h_1+ h_2),
	\end{equation*}
	where 
	\begin{align*}
	& Q^h \eqdef \{x \in \R^d \,|\, d(x,C) \leq h\}  \\
	& Q^{-h} \eqdef \{x \in \R^d \,|\, B(x,h) \subset C\} 
	\end{align*}   
	and $B(x, h) \eqdef \{y \in \R^d\,|\, \|y - x\| \leq h\}$ and $d(x, C) \eqdef \inf_{y \in C}\|y-x\|$. 
\end{lemma}

\subsection{Deviation bound for a random quadratic form}

Consider a Gaussian random vector $\phi \sim \mathcal{N}(0, \Sigma)$ in $\R^d$. Define the following characteristic of $\Sigma$,
\begin{equation*}
\mathtt{p}=\tr(\Sigma), \qquad \mathtt{v}^2= 2\tr(\Sigma^2), \qquad \lambda^{2 *}= \|\Sigma\|_{op}= \lambda_{max}^2(\Sigma)
\end{equation*}
The following theorem provides a Deviation bound for the quadratic form $\|\phi\|$. We use the notation $a \vee b = \max \{a, b\}$. 
\begin{theorem}[Theorem $2.1$ in \citet{spokoiny_sharp_2013}]
	\label{thm:ldevGauss}
	Let $\phi \sim \mathcal{N}(0, \Sigma)$ be a Gaussian random vector in $\R^d$. Then, for every $\mathtt{x}>0$, it holds
	\begin{equation*}
	\P(\|\phi\|^2 > \mathtt{p} + (2 \mathtt{v} \mathtt{x}^{1/2}) \vee (6\lambda^{2 *}\mathtt{x}) ) \leq e^{-\mathtt{x}},
	\end{equation*}
\end{theorem}

\subsection{Comparison of the Euclidean norms of Gaussian vectors}

\begin{proposition}[Pinsker's inequality, \citet{tsybakov_asymptotic_2009} p.88]
	\label{lem:pinsker}
	Let $(\Omega, \mathcal{F})$ denote a measure space and two probability measures $\P_1$ and $\P_2$ on that measure space. Then, it holds
	\begin{equation*}
	\sup_{A \in \mathcal{F}}\bigl |\P_1(A)-\P_2(A) \bigr| \leq \sqrt{\mathcal{KL}(\P_1, \P_2)/2}
	\end{equation*} 
\end{proposition}

\begin{lemma}[Gaussian comparison]
	\label{lem:gaussComp}
	Let $\tilde{\phi}_1 \sim \P_1 \eqdef \mathcal{N}(0, \Sigma_1)$ and $\tilde{\phi}_2 \sim \P_2 \eqdef \mathcal{N}(0, \Sigma_2)$, belong to $\R^d$, and
	\begin{equation}
	\label{eq:covCdts}
	\|\Sigma_2^{-1/2} \Sigma_1 \Sigma_2^{-1/2} - I_{d}\| \leq 1/2, \qquad \tr\bigr( (\Sigma_2^{-1/2} \Sigma_1 \Sigma_2^{-1/2} -I_{d})^2\bigr)\leq \rho_\Sigma^2/2,
	\end{equation}
	for some $\rho_\Sigma^2 \geq 0$. Then it holds
	\begin{equation*}
	\mathcal{KL}(\P_1, \P_2) \leq \rho_\Sigma^2/2
	\end{equation*}
\end{lemma}
\begin{proof}
	Denote $A \eqdef \Sigma_2^{-1/2} \Sigma_1 \Sigma_2^{-1/2}$, then the Kullback-Leibler divergence between $\P_1$ and $\P_2$ is equal to 
	\begin{align*}
	\mathcal{KL}(\P_1, \P_2)&= - 0.5 \log \big( \det (A)\big) + 0.5 \tr(G - I_d) \\
	&= 0.5 \sum_{k=1}^d\big(\lambda_k^2 - \log(\lambda_k^2 +1) \big), 
	\end{align*}
	where $\lambda_1^2 \geq \dots \lambda_d^2$ are the eigenvalues of the matrix $A-I_d$. By conditions of the lemma, $|\lambda_1^2| \leq 1/2$, and it holds:
	\begin{equation*}
	\mathcal{KL}(\P_1, \P_2) \leq 0.5 \sum_{k=1}^d \lambda_k^2= 0.5 \tr\big((A- I_d)^2 \big) \leq \rho_\Sigma^2/2
	\end{equation*}
\end{proof}

For simplicity we provide the next statement for square matrices of size \(d\times d \)
\begin{theorem}[Matrix Bernstein inequality, \citet{Tropp2012}, Theorem 6.1.1]
	\label{theorem:matrix_bernstein}
Consider a finite sequence \(\{S_k\} \) of independent square
random matrices of size \(d\times d\). We assume that
\[
\E S_k = 0, \quad \| S_k \|_{\text{op}} \leq L~\text{for all}~ k.
\]
Define the random matrix
\[
Z = \sum_{k}S_k
\]
and let \(\ups(Z) = \| Z^2\|_{\text{op}} \).
Then for all \(\xx \geq 0 \),
\[
\P \bigl( \| Z\|_{\text{op}} \geq \xx \bigr) \leq 2d\exp\Bigl( \frac{-\xx^2/2}{\ups(Z) + L\xx/3}\Bigr).
\] 
\end{theorem}
The next corollary explains how Theorem~\ref{theorem:matrix_bernstein} is applied in case of empirical covariance matrices.
\begin{corollary}
	\label{corollary:matrix_bernstein}
Let \(\varepsilon_{i} \in \R^d \), \(1\leq i \leq n \) be independent centred random vectors on compact support. Furthermore, we assume that Condition \((Sm^{W}_b) \) holds. Define
\[
\hat{V} \eqdef \sum_{i}\varepsilon_{i}\varepsilon^{\T}_{i}, \quad V \eqdef \E \hat{V} =  \Var\Bigl(\sum_{i}\varepsilon_{i}\varepsilon^{\T}_{i} \Bigr),
\]
\[
Z \eqdef V^{-1/2}\bigl(\hat{V} - V  \bigr) V^{-1/2} = V^{-1/2}\hat{V}V^{-1/2} - I_{d}.
\]
Then for all \(\xx \geq 0 \)
\begin{equation}
\label{eq:matrix_bernstein}
\P \bigl( \| Z \|_{\text{op}} \geq \CONST \xx/\sqrt{n}  \bigr) \leq  2\exp\Bigl( -\bigl(\xx - 6L\log(d) \bigr)/6L\Bigr).
\end{equation}
\end{corollary}
\begin{proof}
Consider a set of vectors \(\zeta_i \), \(i = 1,..., n\), where \( \zeta_i \eqdef V^{-1/2}\varepsilon_i \). Let \( S_i \eqdef \zeta_i \zeta^{\T}_i - \E \zeta_i \zeta^{\T}_i\) and taking into account Condition\((Sm^{W}_b) \) we obtain
\[
\|S_i \|_{\text{op}} = \| \zeta_i \zeta^{\T}_i - \E \zeta_i \zeta^{\T}_i \|_{\text{op}} \leq \| \zeta_i \|^2 \leq \CONST/n.
\]
Let
\[
\ups \eqdef \Bigl\| \sum_{i} \E S^2_i  \Bigr \|_{\text{op}} = 
\Bigl\| \sum_{i} \E(\zeta_i \zeta^{\T}_i)^2 - \bigl( \E \zeta_i \zeta^{\T}_i  \bigr)^2  \Bigr \|_{\text{op}}.
\]
A rough bound on \(\ups \) can be obtained using assumption, that \(S \eqdef \E \sum_{i}S_i = I_{d} \):
\[
\ups \leq \CONST/n \Bigl\|\sum_{i}\E \zeta_i \zeta^{\T}_i \Bigr \|_{\text{op}} = \CONST/n.
\]
Now the result of Theorem~\ref{theorem:matrix_bernstein} implies for \(Z \)
\[
\P \bigl( \| Z \|_{\text{op}} \geq \CONST \xx/\sqrt{n}  \bigr) \leq  2d\exp\Bigl( \frac{-\xx^2/2}{\CONST/n + L\xx/3}\Bigr).
\]
Assuming that \(n \) is quite large, we can estimate
\[
\P \bigl( \| Z \|_{\text{op}} \geq \CONST \xx/\sqrt{n}  \bigr) \leq  2\exp\Bigl( -\bigl(\xx - 6L\log(d) \bigr)/6L\Bigr).
\]
\end{proof}

\begin{lemma}[Sub-exponential tail bound]
	\label{lemma:sub_exp}
Suppose that \(X \) is subexponential, i.d. \((ED^{W}_{b}) \)-like condition is fulfilled.
Then
\[
\P \bigl( X - \E X \geq \xx  \bigr)
\leq
\begin{cases}
e^{-\xx^2 / 2\nu^2}, &~\text{if}~ 0 \leq \xx \leq \nu^2/b,\\
e^{-\xx^2 / 2b}, &~\text{if}~ \xx > \nu^2/b.
\end{cases}
\]
\end{lemma}

\bibliographystyle{imsart-nameyear}
\bibliography{references}

\end{document}